\documentclass[11pt]{article}

\usepackage[text={6.5in,9in},centering]{geometry}

\usepackage{amsmath,amssymb,amsfonts,amsthm}
\usepackage{bm}
\usepackage{mathrsfs}
\usepackage[skip=16pt]{caption}   
\usepackage{graphicx}
\graphicspath{{./}{../pdf/}{../jpeg/}}
\DeclareGraphicsExtensions{.pdf,.jpeg,.png}

\usepackage{array,multirow,booktabs,setspace}
\usepackage{tabularx}
\usepackage{lipsum}         
\usepackage{eso-pic}        
\usepackage{tikz}           
\usetikzlibrary{calc}       

\AddToShipoutPictureBG{%
	\begin{tikzpicture}[remember picture, overlay]
		\foreach \offset in {-10cm, 0cm, 10cm} {
			\node[rotate=45, scale=4, opacity=0.2, text=gray] 
			at ($(current page.center) + (0cm, \offset)$) 
			{PASCAL-HESSIAN CRITERION};
		}
	\end{tikzpicture}%
}
\usepackage{xcolor,soul,framed}
\usepackage{enumitem}
\usepackage{url}
\usepackage{algorithm,algorithmic}
\usepackage{cite}
\usepackage{hyperref}

\allowdisplaybreaks[4]

\newtheorem{theorem}{Theorem}
\newtheorem{lemma}{Lemma}

\newtheorem{definition}{Definition}

\newtheorem{remark}{Remark}

\newenvironment{keywords}
  {\small\noindent\textbf{Keywords: }}
  {\par\medskip}

\DeclareMathOperator{\spanop}{span}

\DeclareMathOperator{\col}{col}

\begin{document}

\title{\LARGE\bfseries Simple Verification and Implementation of Observer Error Dynamics Linearization: A Pascal's Triangle--Hessian Matrix Criterion}

\author{Haotian~Xu, Xinquan~Shao, Shuai~Liu,~	
\thanks{This work is supported by National Natural Science Foundation of China (No.61533013, 61633019, 61903290), Funded by China Postdoctoral Science Foundation (No.2022M721932). (Corresponding author: Shuai Liu)}
\thanks{Haotian Xu is with the School of Artificial Intelligence, Shandong University, Jinan 250061, China; and also with Key Laboratory of Machine Intelligence and System Control, Ministry of Education, China~(e-mail:  \href{mailto:xuhaotian\_1993@126.com}{xuhaotian\_1993@126.com}).}
\thanks{Xinquan Shao and Shuai Liu are with the School of Control Science and Engineering, Shandong University, Jinan 250061, China; and also with Key Laboratory of Machine Intelligence and System Control, Ministry of Education, China~(e-mail: \href{mailto:liushuai@sdu.edu.cn}{liushuai@sdu.edu.cn}).}}

\date{}
\maketitle
\vspace{-6pt}
\begin{center}
	\rule{0.8\textwidth}{0.4pt}
\end{center}
\vspace{6pt}

\begin{abstract}
The classical theory of nonlinear observer error linearization---the nonlinear observer canonical form---has attracted sustained attention since its inception in the 1980s. Under the existing theoretical framework, verifying the necessary and sufficient condition for a nonlinear system to achieve observer error linearization requires extensive Lie bracket involutivity checks, while constructing the output-dependent univariate functions in the canonical form demands solving systems of partial differential equations and performing nonlinear coordinate transformations. Neither verification nor construction admits a systematic, efficient implementation, severely limiting the applicability of the theory to high-dimensional systems. To address this issue, inspired by the definition of high-order fully measured systems, this paper proposes the Pascal-Hessian condition for single-output systems. This condition equivalently converts the necessary and sufficient condition for observer error linearization into a structural test on the Hessian matrix of the nonlinear term in high-order fully measured systems: the coefficients in the upper-left corner of Hessian matrix form a Pascal's triangle, while the lower-right corner vanishes identically. Simultaneously, we provide explicit integral formulas for all output-dependent univariate functions in the canonical form, eliminating the need to solve partial differential equations. Compared with the existing theory, our method reduces the computational complexity of condition verification from $O(n^4)$ to $O(n^2)$, and replaces the intricate process of solving partial differential equations with explicit indefinite integral for canonical form construction. We further extend the result to multi-output systems with equal observability indices. This work reveals the intrinsic mathematical connection between the nonlinear observer canonical form and Pascal's triangle, providing a completely formulaic and directly operable pathway to observer error linearization. Numerical examples validate the effectiveness of the proposed method.
\end{abstract}


\section{Introduction}\label{sec1}
\subsection{Problem Background}\label{sec1.1}
Consider the nonlinear system
\begin{align}
	&\dot{\chi}=f(\chi),\label{sys11}\\
	&y=h(\chi),\label{sys12}
\end{align}
where $\chi=\col\{\chi_1,\ldots,\chi_n\}\in\mathbb{R}^n$ and $y\in\mathbb{R}$ are the system state and measurement output (notations for $\mathbb{R}$, $\mathbb{R}^n$, $\col\{\cdot\}$, etc.\ are explained at the beginning of Section~\ref{sec2}), $f(\chi)=\col\{f_1(\chi),\ldots,f_n(\chi)\}$ is an $n$-dimensional smooth nonlinear vector-valued function, and $h(\chi)$ is a smooth function. When $f(\chi)$ and $h(\chi)$ satisfy the necessary and sufficient conditions for the nonlinear observer canonical form (which will be detailed in Section~\ref{sec2}) , system~(\ref{sys11})--(\ref{sys12}) can be transformed via a nonlinear transformation $z=\Phi(\chi)$ into the following observer canonical form (OCF)~\cite{Isidori1995}:
\begin{align}
	&\begin{bmatrix}
		\dot{z}_1\\\dot{z}_2\\\vdots\\\dot{z}_n
	\end{bmatrix}=\begin{bmatrix}0&\cdots&0&0\\
								 1&\cdots&0&0\\
								 \vdots&\ddots&\vdots&\vdots\\
								 0&\cdots&1&0
     \end{bmatrix}\begin{bmatrix}
     z_1\\z_2\\\vdots\\z_n
\end{bmatrix}+\begin{bmatrix}
\alpha_1(y)\\\alpha_2(y)\\\vdots\\\alpha_n(y)
\end{bmatrix},\\
&y=z_n,
\end{align}
where, for any $i=1,\ldots,n$, $\alpha_i(y)$ is a smooth univariate function. Let $z=\col\{z_1,\ldots,z_n\}$, $\alpha(y)=\col\{\alpha_1(y),\ldots,\alpha_n(y)\}$, and
\begin{align}
	A=\begin{bmatrix}0&\cdots&0&0\\
		1&\cdots&0&0\\
		\vdots&\ddots&\vdots&\vdots\\
		0&\cdots&1&0
		\end{bmatrix}, C=\begin{bmatrix}0&\cdots&0&1\end{bmatrix}.
\end{align}
We obtain the compact form
\begin{align}
	\dot{z}=Az+\alpha(y),~y=Cz.\label{sys2}
\end{align}
Based on system~(\ref{sys2}), one can readily design an observer gain matrix $L$ such that the observer
\begin{align}
	\dot{\hat{z}}=A\hat{z}+L(y-C\hat{z})+\alpha(y)
\end{align}
has linear error dynamics. Since its inception, the above knowledge has been one of the powerful tools in the field of nonlinear observer design.

In 2020, Duan et al established the high-order fully measured system (FMS) theory based on the observation that most practical physical systems can be described in the high-order form $x_1^{(n)}=p(x)$, where $x=\operatorname{col}\{x_1,\ldots,x_n\}\in\mathbb{R}^n$ is the state vector, the superscript $(n)$ denotes the $n$-th order time derivative, and $p(x)$ is a nonlinear function \cite{Duan2020III,Duan2024Overview}. Specifically, in \cite{Duan2020III}, Duan derived from the OCF the standard form for a high-order system to achieve observer error linearization---the high-order fully measured form: for the high-order system
\begin{align}
	x_1^{(n)}=p(x),~y=x_1\label{sys3}
\end{align}
if there exists a family of smooth univariate functions $B_r(\cdot)$, $r=0,\ldots,n-1$, such that $p(x)$ takes the form
\begin{align}
	p(x)=\sum_{r=0}^{n-1}B_r^{(r)}(x_1),\label{fms-p}
\end{align}
then the nonlinear system~(\ref{sys3}) is called a high-order fully measured system, and the nonlinear function $\alpha(y)$ in the observer canonical form can be directly obtained from the high-order fully measured form as $\alpha_i(y)=B_{i-1}(y)$. For practical systems that can be expressed in high-order form, the FMS theory is considerably simpler than the traditional OCF approach for obtaining a nonlinear observer with linear error dynamics (as will be introduced in Appendix~\ref{app1}, the traditional method for solving OCF requires extensive Lie derivative and Lie bracket computations, followed by solving partial differential equations (PDEs) and performing nonlinear transformations on the original system).

The central issue, however, is that while many practical systems can be expressed in the high-order form (\ref{sys3}), one must still determine whether $p(x)$ satisfies the FMS condition (\ref{fms-p}) and how to obtain the corresponding $B_r(y)$ in~(\ref{fms-p}). Unfortunately, neither the original FMS theory~\cite{Duan2020III} nor the subsequent works~\cite{Zhang2024Stabilization,Zhao2023Observer} provide a solution to this problem.

Thus, to this day, verifying whether~(\ref{fms-p}) holds still requires relying on the OCF verification method (checking $n(n-1)/2$ Lie bracket conditions); and solving for $B_r(y)$ can only be done through the OCF computation (i.e., solving PDEs and performing nonlinear transformations, among other steps). Clearly, this significantly undermines the practical utility of the FMS theory for verification and construction purposes---it was originally proposed to circumvent the cumbersome verification and solution procedures of the OCF theory, yet the lack of an effective quantifiable verification method for condition~(\ref{fms-p}) forces one to continue using the full OCF verification and solution machinery.

It is therefore evident that establishing a quantifiable necessary and sufficient condition for~(\ref{fms-p}) based on the high-order system form~(\ref{sys3}), and proposing a method to directly obtain $B_r(y)$ from $p(x)$, is of great importance. Moreover, and more significantly, if a verification and solution method based on $p(x)$ can be established, then general nonlinear systems~(\ref{sys11})--(\ref{sys12}) can also avoid the cumbersome verification and solution methods of the OCF theory, because a simple nonlinear transformation can readily convert~(\ref{sys11})--(\ref{sys12}) into the high-order form~(\ref{sys3}) (this transformation will be introduced in Section~\ref{sec2}). Hence, if one could directly verify, based on $p(x)$, whether a system is FMS and directly extract the $B_r(y)$ therein, it would be possible to dramatically simplify the difficulty of verifying and implementing observer error linearization under the existing framework.

This paper will solve this problem. Perhaps surprisingly, the simple method behind this classical theory (observer error linearization) dating back to the 1980s lies hidden in Pascal's triangle, which was born as early as 800 years ago (Pascal's triangle, also known as the Yang Hui triangle, was first recorded in the Chinese treatise \emph{Nine Chapters on the Mathematical Art} in 1261~AD)! The main body of this paper will reveal the fascinating connection among observer error linearization, FMS theory, and Pascal's triangle.

\subsection{Related Works}

State estimation for nonlinear systems is a central problem in control theory. Over the past half-century, nonlinear observer research has developed along multiple parallel tracks: the extended Kalman filter achieves recursive estimation via local linearization~\cite{gelb1974,10264154}; high-gain observers~\cite{khalil2017,Li2025PrescribedTime,Farza2025SATFHGO,10497867} and sliding mode observers~\cite{slotine1987,spurgeon2008,edwards1998,9536496} suppress nonlinearities and uncertainties through large-magnitude gains and discontinuous switching terms, respectively; moving horizon estimation formulates state estimation as a finite-horizon constrained optimization problem~\cite{rao2003,10453955,Xie2024BackAndForth}; in recent years, data-driven and learning-based observer methods have provided new pathways for systems that are difficult to model accurately~\cite{tang2024,Marani2025DeepLearningKKL,Zhou2024HealthStatus}. Among these methods, observer error linearization occupies a distinct position: it not only establishes the observability rank condition for nonlinear systems but also constructs OCF for autonomous and affine nonlinear systems, endowing nonlinear systems with completeness in observability theory comparable to that of linear systems.

The mathematical foundation of observer error linearization framework was laid by Hermann and Krener~\cite{Hermann1977}, who established the observability rank condition. Krener and Isidori~\cite{Krener1983} first gave the necessary and sufficient geometric condition for a single-output autonomous system to be transformable into the OCF in 1983; subsequently, Krener and Respondek~\cite{Krener1985} extended this theory to multi-output cases. In 1989, Xia and Gao~\cite{Xia1989} further refined the necessary and sufficient conditions for the multi-output cases. This work became the standard framework for subsequent works in this direction \cite{Hou1999,Phelps1992}. Boutat et al.~\cite{Boutat2006CDC} first proposed OCF concept and its constructive algorithm of the multi-output affine nonlinear system, and subsequently further provided the necessary and sufficient conditions together with an explicit algorithm for transforming affine nonlinear systems into the OCF~\cite{Boutat2009Automatica}. For systems that do not satisfy the OCF conditions, works such as~\cite{Boutat2011IJC,Tami2013Automatica,Boutat2015} have proposed extended nonlinear observer canonical forms (EOCF) based on dynamic extension methods.

In the partial observability setting, Boutat et al.~\cite{Boutat2016Automatica} systematically studied the partial observer canonical form (POCF) problem, and Saadi et al.~\cite{Saadi2020TAC} subsequently extended the constructive algorithm for POCF to multi-output systems. For discrete-time systems, Lee~\cite{Lee2011SCL} pioneered the conversion of Lie bracket involutivity tests into directly computable algebraic rank conditions, later extending the result to multi-output configurations~\cite{Lee2017TAC}; related work has also been generalized to discrete-time generalized OCF~\cite{Hong2008SCL} and reduced-order observer canonical forms~\cite{Boutat2012SCL}.

Although the OCF theory for nonlinear systems has by now been developed to a relatively complete state, the traditional verification and construction procedures remain extremely cumbersome, to the extent that full computer implementation is challenging. This sets a high theoretical barrier for practical application. Consequently, recent research focus has gradually shifted from formulating purely geometric conditions toward verifiability and computational simplification. Lee \cite{Lee2017} computed the verifiable Lie bracket conditions for multioutput systems and also worked on converting geometric conditions into algebraically more verifiable forms \cite{Lee2011SCL,Lee2017TAC}; Xu and Wang~\cite{Xu2022IJCAS} proposed a two-step method for computing partial observer canonical forms. The monograph by Boutat and Zheng~\cite{Boutat2021Book} systematically summarizes the theoretical achievements up to 2021. These advances have collectively pushed observer error linearization from a geometrically elegant but computationally cumbersome framework toward a design tool with clearer conditions and stronger verifiability.

Nevertheless, the various simplification methods mentioned above still rely fundamentally on verifying a large number of Lie bracket conditions, and thus do not alleviate the inherent difficulty of the necessary and sufficient condition verification. To overcome this limitation, some scholars have sought breakthroughs from the level of system structure. In 2024, Xu and Liu et al.~\cite{Xu2025Distributed}, inspired by the definition of high-order fully measured systems, proposed a class of nonlinear systems directly transformable into FMS form. Liu et al.~\cite{Liu2025Nonlinear} extended this idea to more general multi-input multi-output settings. Unlike~\cite{Xu2025Distributed}, which restricts nonlinear terms to combinations of derivatives of the output function,~\cite{Liu2025Nonlinear} allows quadratic terms and cross-subsystem couplings. Through integral-type coordinate transformations that eliminate undesired coupling terms and the coupled auxiliary dynamics method~\cite{Wang2020} that absorbs residual coefficients, the original system is ultimately converted into the extended nonlinear observable canonical form (EOCF).

While~\cite{Xu2025Distributed,Liu2025Nonlinear} have established ingenious methods for specific system structures that circumvent the cumbersome Lie bracket verification and PDE solving of the OCF, these approaches do not generalize to arbitrary nonlinear systems. This paper continues the line of research initiated by~\cite{Xu2025Distributed}, seeking, for general nonlinear systems, an alternative to the traditional OCF necessary and sufficient conditions. By establishing the connection among FMS, Pascal's triangle, and observer error linearization, this paper provides a systematic simplification pathway for OCF verification and construction. The main contributions are as follows.

\begin{enumerate}[label=\arabic*)]
\item For the general high-order system~(\ref{sys3}), we establish a directly verifiable necessary and sufficient condition based on the function $p(x)$, termed the Pascal-Hessian condition. It states that system~(\ref{sys3}) admits the high-order fully measured form if and only if the lower-right corner of the Hessian matrix of $p(x)$ vanishes identically, while the coefficients in the upper-left corner form a Pascal triangle.

\item The Pascal-Hessian condition is fully equivalent to the traditional OCF verification condition. In algorithmic implementations, checking this condition requires only $O(n^2)$ operations, compared with $O(n^4)$ for the classical necessary and sufficient condition, where $n$ is the system dimension.

\item Based on $p(x)$, we propose a formulaic construction method for the functions $B_r(y)$: all such functions are obtained directly via explicit integral formulas, requiring neither PDE solving nor complex nonlinear transformations~\cite{Isidori1995,Boutat2021Book}. Unlike the works~\cite{Xu2025Distributed,Liu2025Nonlinear}, which also avoid Lie brackets and PDEs but are confined to specific system structures, our result holds for general nonlinear systems.

\item The above results extend directly to general systems~(\ref{sys11})--(\ref{sys12}) and to multi-output systems with equal observability indices.
\end{enumerate}

The remainder of this paper is organized as follows. First, Section~\ref{sec2} provides preliminaries and the problem description. Section~\ref{sec3} presents the Pascal's Triangle--Hessian matrix criterion method for single-output systems to achieve observer error linearization, together with a detailed comparison with the traditional OCF theory. Section~\ref{sec4} describes the Pascal-Hessian condition for multi-output systems with equal observability indices. Section~\ref{sec:example} demonstrates the effectiveness of the proposed method through three examples. Finally, Section~\ref{sec6} summarizes the paper and discusses future work.

\section{Preliminaries and Problem Description}\label{sec2}

This paper adopts the following notational conventions. Let $\mathbb{R}^n$ denote the $n$-dimensional real vector space, with $I_n$ the $n\times n$ identity matrix and $A^{\top}$ the transpose of $A$. We write the ceiling function, the Euclidean inner product, and the Kronecker symbol as $\lceil\cdot\rceil$, $\langle\cdot,\cdot\rangle$, and $\delta_{i,j}$ (with $\delta_{i,j}=1$ if $i=j$, and $0$ otherwise), respectively. The binomial coefficient is $\binom{n}{k} = \frac{n!}{k!(n-k)!}$. The operator $\operatorname{col}\{\cdot\}$ stacks its arguments into a column vector, whereas $\operatorname{span}\{\cdot\}$ spans the distribution or codistribution of its entries. Finally, $x^{(n)}$ denotes the $n$-th order time derivative of $x$.

\subsection{Differential Geometry}
\begin{definition}
	 Given an $n$-dimensional smooth manifold $\mathcal{M}$, for any point $\chi$ on $\mathcal{M}$ and a neighborhood $\mathcal{U}$ containing $\chi$, if there exist smooth functions $f_1,\ldots,f_n$ defined on $\mathcal{U}$ satisfying:
		\begin{equation}\label{eq:ch2-vector field}
			f(\chi)=\sum_{i=1}^nf_i(\chi)\frac{\partial}{\partial \chi_i}, ~~\forall \chi\in\mathcal{U},
		\end{equation}
		then $f$ is called a smooth vector field.
\end{definition}
In this definition, $\frac{\partial}{\partial \chi_i},~i=1,\ldots,n$ are the basis vectors in the $\chi$ coordinates, and the functions $f_i(\chi),~i=1,\ldots,N$ are the coefficients of the vector field in the $\chi$ coordinates. Therefore, the vector field $f(\chi)$ can also be written in vector form depending only on the coefficients, namely $f(\chi)=\col\{f_1(\chi),\ldots,f_n(\chi)\}$. Based on this definition, the smooth vector-valued function $f(\chi)$ in the nonlinear system~(\ref{sys11}) is generally called a smooth vector field. For an output function $h(\chi)$, we call $L_fh(\chi)=\frac{\partial h(\chi)}{\partial \chi^\top}f(\chi)$ the Lie derivative of the function $h$ along the direction of the vector field $f$. For any positive integer $k\geq 2$, higher-order Lie derivatives are defined as $L_f^kh(\chi)=L_f(L_f^{k-1}h(\chi))$. For two vector fields $f(\chi)$ and $g(\chi)$, the Lie bracket $[f,g]$ is defined as 
\begin{align}\label{Lie-bracket}
	[f,g]=fg-gf=\sum_{i=1}^n\left(\frac{\partial g}{\partial \chi^\top}f-\frac{\partial f}{\partial \chi^\top}g\right)_i\frac{\partial}{\partial \chi_i}.
\end{align}

Corresponding to vector fields are covector fields. A covector field ${\rm d}h(\chi)$ in a neighborhood $\mathcal{U}$ of a point $\chi$ can be expressed as
\begin{align}
	{\rm d}h(\chi)=\begin{bmatrix}
		\frac{\partial h(\chi)}{\partial \chi_1}&\cdots&\frac{\partial h(\chi)}{\partial \chi_n}
	\end{bmatrix}\col\{{\rm d}\chi_1,\ldots,{\rm d}\chi_n\},~\forall \chi\in\mathcal{U},
\end{align}
where $\frac{\partial h(\chi)}{\partial \chi_1}, \ldots, \frac{\partial h(\chi)}{\partial \chi_n}$ are the coefficients of the covector field ${\rm d}h(\chi)$ with respect to the basis ${\rm d}\chi_1,\ldots,{\rm d}\chi_n$. For the $f(\chi)$ and $h(\chi)$ in the nonlinear system~(\ref{sys11})--(\ref{sys12}), we call
\begin{align}
	\Delta^\perp=\spanop\{{\rm d}h(\chi),~{\rm d}L_fh(\chi),\ldots,{\rm d}L_f^{n-1}h(\chi)\}
\end{align}
a codistribution. The dimension of the codistribution $\Delta$ is defined as the rank of the matrix formed by the coefficient vectors of the $n$ covector fields (${\rm d}h(\chi)$, ${\rm d}L_fh(\chi)$, $\ldots$, ${\rm d}L_f^{n-1}h(\chi)$) that span the codistribution.

If $\Delta^\perp$ has dimension $n$ in the region $\mathcal{U}$, then the system~(\ref{sys11})--(\ref{sys12}) is said to be completely observable in the region $\mathcal{U}$.

\subsection{Two Diffeomorphism Transformations for Nonlinear Observers}

In Section~\ref{sec1}, we have already presented the concrete form of the observer canonical form~(\ref{sys2}). Below, we introduce, via a lemma, an important result from literature \cite{Isidori1995} concerning the observer canonical form and observer error dynamics linearization.
\begin{lemma}\label{lem:ocf}
	In the region $\mathcal{U}$, there exists a diffeomorphism (a map is a diffeomorphism if and only if its inverse map exists and both the map and its inverse are smooth) $z=\Phi(\chi)$ such that the system~(\ref{sys11})--(\ref{sys12}) can be transformed into the OCF~(\ref{sys2}) if and only if: 1) the codistribution $\Delta^\perp$ has dimension $n$ in $\mathcal{U}$; 2) for any $i,j=1,\ldots,n$, $[\tau_i,\tau_j]=0$, where the vector fields $\tau_i$ are the solutions of the system of linear equations
	\begin{align}
		\frac{\partial L_f^rh(\chi)}{\partial \chi^\top}\cdot\tau_1=\delta_{r,n-1},~r=0,\ldots,n-1\label{ocf}
	\end{align}
	and $\tau_k=[\tau_{k-1},f]$.
\end{lemma}
The above lemma gives the necessary and sufficient condition for transforming $\chi$ to the OCF ($z$). The concrete construction method will be introduced in Appendix~\ref{app1}.

The subsequent arguments of this paper also require another set of coordinate transformations---the transformation from $\chi$ to the observer form ($x$). In $\mathcal{U}$, via the nonlinear diffeomorphism $x=\Psi(\chi)$:
\begin{align}
	x_1=h(\chi),~x_2=L_fh(\chi),~\ldots,~x_n=L_f^{n-1}h(\chi)\label{observerform-trans}
\end{align} 
one can easily transform~(\ref{sys11})--(\ref{sys12}) into an $n$-dimensional single-output nonlinear system:
\begin{align}\label{eq:chain}
	\dot{x}_1 =& x_2,\quad \dot{x}_2 = x_3,\quad \ldots,\quad \dot{x}_{n-1} = x_n,\notag\\ \dot{x}_n =& p(x_1, \ldots, x_n).
\end{align}
It is not difficult to see that this system is equivalent to the form of the high-order system~(\ref{sys3}) (literature \cite{Krener1985} the form~(\ref{eq:chain}) as the observer form). In the $x$ coordinates of the observer form, we define the vector field
\begin{equation}
	f_p(x) = \col\{x_2, x_3, \ldots, x_n, p(x)\}.
\end{equation}
The remainder of this paper will primarily focus on the system~(\ref{sys3}) or~(\ref{eq:chain}).

\subsection{Problem Description}

The literature~\cite{Duan2020III} proved that if the high-order system~(\ref{eq:chain}) can be expressed in the fully measured form:
\begin{align}
	x_1^{(n)}=\sum_{r=0}^{n-1}B_r^{(r)}(x_1),~y=x_1\label{fms}
\end{align}
then based on the diffeomorphism:
\begin{align}
	&z_n=x_1,\label{transform1}\\
	&z_k=x_{n-k+1}-\sum_{r=k}^{n-1}B_r^{(r-k)}(x_1),~k=1,\ldots,n-1,\label{transform2}
\end{align}
one can directly obtain the observer canonical form~(\ref{sys2}), where $\alpha_i(y)=B_{i-1}(y)$. One can thus construct a nonlinear observer with linear error dynamics. However, as noted in Section~1.1, given a high-order system $x_1^{(n)}=p(x)$, we have no way of determining whether $p(x)$ satisfies the condition~(\ref{fms-p}), nor can we directly extract the key elements $B_r(x_1)$, $r=0,\ldots,n-1$, of the fully measured form~(\ref{fms}) from $p(x)$. And if one resorts to the condition~(\ref{ocf}) and the methods given in Appendix~\ref{app1} to verify and obtain the OCF, the process is extremely cumbersome. Therefore, the objective of this paper is to establish, based on $p(x)$, a necessary and sufficient condition equivalent to~(\ref{ocf}) that enables the system~(\ref{eq:chain}) to possess the fully measured form~(\ref{fms}), and further to establish, via $p(x)$, a formulaic method for obtaining the nonlinear functions $B_r(x_1)$, $r=0,\ldots,n-1$.

\section{Pascal's Triangle--Hessian Matrix Criterion Method}\label{sec3}

Subsection~\ref{sec3.1} of this section will present the Pascal-Hessian condition based on $p(x)$ (this condition is the necessary and sufficient condition for the system~(\ref{eq:chain}) to be transformable into the high-order fully measured form~(\ref{fms})) together with the solution method for the functions $B_r(x_1),~r=0,\ldots,n-1$; Subsection~\ref{sec3.2} will prove the conclusions of~\ref{sec3.1}; subsequently, in Subsection~\ref{sec3.3}, we present a detailed comparison of the advantages and disadvantages of the two conditions and the corresponding OCF solution methods.

\subsection{Pascal-Hessian Condition and the Main Theorem}\label{sec3.1}

First, we directly give the definition of the Pascal-Hessian condition.
\begin{definition}[Pascal-Hessian Condition]\label{def:PH}
	A function $p(x)$ is said to satisfy the Pascal-Hessian condition if and only if for all $1 \leq i, j \leq n$, the following holds:
	\begin{equation}\label{eq:PH}
		\frac{\partial^2 p}{\partial x_i \partial x_j} =
		\begin{cases}
			\displaystyle \binom{i+j-2}{i-1} \cdot \frac{\partial^2 p}{\partial x_1 \partial x_{i+j-1}}, & i+j-1 \leq n, \\[1.5em]
			0, & i+j-1 > n.
		\end{cases}
	\end{equation}
	Equivalently, the Hessian matrix $\mathcal{H}(p)$ of $p(x)$ possesses a Pascal triangle structure. By Pascal triangle structure we mean that the upper-left corner (including the anti-diagonal) of $\mathcal{H}(p)$ forms a Pascal triangle, while the lower-right corner is identically zero.
\end{definition}

We will use 3rd-order, 4th-order, and 5th-order systems to exhibit the concrete form of the Hessian matrix with Pascal triangle structure. Taking $n=4$ as an example, the Hessian matrix $\mathcal{H}(p)|_{n=4}$ with Pascal triangle structure is:
\begin{equation*}
\mathcal{H}(p)|_{n=4} = 
\begin{bmatrix}
1\cdot \partial_1\partial_1 p & 1\cdot \partial_1\partial_2 p & 1\cdot \partial_1\partial_3 p & 1\cdot \partial_1\partial_4 p  \\
1\cdot \partial_1\partial_2 p & 2\cdot \partial_1\partial_3 p & 3\cdot \partial_1\partial_4 p & 0  \\
1\cdot \partial_1\partial_3 p & 3\cdot \partial_1\partial_4 p & 0 & 0  \\
1\cdot \partial_1\partial_4 p & 0 & 0 & 0 

\end{bmatrix},
\end{equation*}
where $\partial_i\triangleq \frac{\partial}{\partial x_i}$ (so that $\partial_i\partial_jp=\frac{\partial^2p}{\partial x_j\partial x_i}$). From~$\mathcal{H}(p)|_{n=4}$ one can see that in the matrix, bounded by the anti-diagonal, the lower-right entries are all zero; whereas in the upper-left part, every entry is a multiple of the entry in the first column of its lower-left, and this multiple is one of the binomial coefficients (for example, the coefficients of the four entries on the anti-diagonal are $1, 3, 3, 1$, respectively); all the coefficients in the upper-left part of the matrix, taken together, exactly constitute a Pascal triangle. We place the Hessian matrices with Pascal triangle structure for $n=3$ and $n=5$ in Fig.~\ref{yanghui1} for the reader's further understanding.

\begin{figure}[!t]
	\centering
	\includegraphics[width=14cm]{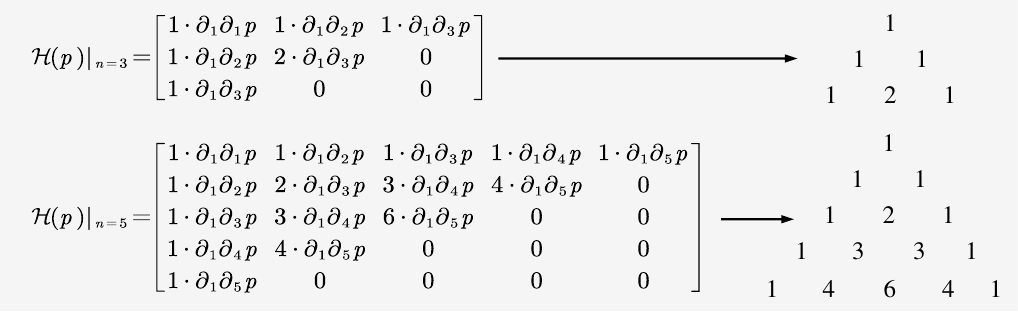}\\
	\caption{Illustration of the Pascal-Hessian condition for 3rd- and 5th-order systems.}\label{yanghui1}
\end{figure}

At this point, we can state the main theorem of this paper as follows.

\begin{theorem}\label{thm:PH}
	Given the high-order system~(\ref{sys3}), there exists a family of smooth univariate functions:
	\begin{align}
		B_0(x_1) &= p(x_1, 0, \ldots, 0), \label{eq:B0} \\
		B_r(x_1) &= \int \frac{\partial p}{\partial x_{r+1}}(x_1, 0, \ldots, 0) \, {\rm d}x_1,\quad r = 1, \ldots, n-1, \label{eq:Br}
	\end{align}
	such that the high-order system can possess the high-order fully measured form~(\ref{fms}) (equivalently, $p(x)=\sum_{r=0}^{n-1} D^r B_r(x_1)$) in the region $\Psi(\mathcal{U})$ (where $\Psi$ is the diffeomorphism defined in~(\ref{observerform-trans})) if and only if $p(x)$ satisfies the Pascal-Hessian condition in the region $\Psi(\mathcal{U})$. 
\end{theorem}

Theorem \ref{thm:PH} gives a necessary and sufficient condition for a high-order system to admit a high-order full-measurement form, namely the Pascal--Hessian condition. The proof of this theorem is rather involved; to preserve the continuity of the main narrative, we will present its proof in subsection \ref{sec3.2}.

Theorem~\ref{thm:PH} simultaneously provides an operable verification method and a construction method. For verification, after computing the Hessian matrix of $p(x)$, one only needs to check whether its lower-right corner is identically zero and whether the upper-left coefficients conform to Pascal's triangle---since the second-order derivatives of linear terms are zero, the verification only needs to be performed on the nonlinear terms in $p(x)$. For construction, $B_r(x_1)$ is directly given by the indefinite integral~(\ref{eq:B0})--(\ref{eq:Br}) on the section $\Sigma = \{x\mid x_2=\cdots=x_n=0\}$; substituting into formulas~(\ref{transform1})--(\ref{transform2}) immediately yields the OCF coordinate transformation and the output injection functions $\alpha_i(y)=B_{i-1}(y)$. The entire process requires no computation of $\tau$ vector fields, no PDE solving, and no matrix inversion.

\subsection{Proof of Theorem~\ref{thm:PH}}\label{sec3.2}

Before we begin, we first give the definition of the total derivative operator and prove two lemmas related to it. We define the total derivative operator $D$ as:
	\begin{equation}\label{eq:D}
		D = \sum_{\ell=1}^{n-1} x_{\ell+1} \frac{\partial}{\partial x_\ell}.
	\end{equation}
According to this definition, $D$ is a vector field and can be expressed in vector form by its coefficients as 
\begin{align}\label{field-D}
	D=\begin{bmatrix}
		x_2,\ldots,x_{n},0
	\end{bmatrix}^\top.
\end{align}

\begin{lemma}\label{lem-1}
	Let $F(x_1)$ be a smooth univariate function. Then for any $r=0,\ldots,n-1$, $D^r F(x_1)$ depends at most on $x_1,\ldots,x_{r+1}$.   
\end{lemma}
\begin{proof}
	We proceed by induction on $r$. For the case of $r=0,1$: $D^0F(x_1)=F(x_1)$ depends only on $x_1$; $DF(x_1)=x_2\partial_1F(x_1)$ depends only on $x_1,~x_2$. Then, we assume $D^rF(x_1)$ depends at most on $x_1,\ldots,x_{r+1}$ for some $r\geq 1$. Hence, by the definition of $D$,
	\begin{align*}
	D^{r+1}F(x_1)=D(D^rF(x_1))=\sum_{\ell=1}^{n-1}x_{\ell+1}\frac{\partial}{\partial x_\ell}D^rF(x_1).
	\end{align*}
	By the induction hypothesis, $\frac{\partial}{\partial x_\ell}D^rF(x_1)=0$ for all $\ell>r+1$. Hence each nonzero term $x_{\ell+1}\frac{\partial}{\partial x_\ell}D^rF(x_1)$ in the sum satisfies $\ell\leq r+1$, so $x_{\ell+1}$ involves at most $x_{r+2}$. Therefore $D^{r+1}F(x_1)$ depends at most on $x_1,\ldots,x_{r+2}$.
	
	By induction, the claim holds for all $r=0,\ldots,n-1$.
\end{proof}

\begin{lemma}\label{lem-D}
	For any smooth univariate function $F(x_1)$ and integer $m \geq 0$, $1 \leq i \leq n$, we have:
	\begin{equation}\label{eq:op_deriv}
		\partial_i D^m F(x_1) = \binom{m}{i-1} D^{m-i+1} \partial_1F(x_1),
	\end{equation}
	with the convention that when $i-1 > m$, $D^{m-i+1}=0$ and $\binom{m}{i-1} = 0$.
\end{lemma}
\begin{proof}
	We discuss two cases.
	
	1) Consider the case $i-1\leq m$. First, we prove $[\partial_i,D]=\partial_{i-1},~i=2,\ldots,n$ and $[\partial_1,D]=0$. Indeed, we know that the total derivative operator is a vector field and its coefficient vector in $x$-coordinates is~(\ref{field-D}); moreover, $\partial_i$ is also a vector field and its coefficient vector is $e_i$, where $e_i$ is the $n$-dimensional column vector with only the $i$-th entry equal to $1$ and all other entries equal to $0$. Hence, using the Lie bracket formula~(\ref{Lie-bracket}), we obtain
	\begin{align}
		[\partial_i,D]=\sum_{j=1}^n\left(\frac{\partial}{\partial x^\top}\col\{x_2,\ldots,x_n,0\}e_i\right)_j\frac{\partial}{\partial x_j}.
	\end{align}
	Note that when $i=1$, $\frac{\partial}{\partial x^\top}\col\{x_2,\ldots,x_n,0\}e_i=0$, hence $[\partial_1,D]=0$; moreover, when $i=2,\ldots,n$, we have
	\begin{align}
		\frac{\partial}{\partial x^\top}\col\{x_2,\ldots,x_n,0\}e_i=e_{i-1}.
	\end{align}
	Therefore, $[\partial_i,D]=\frac{\partial}{\partial x_{i-1}}=\partial_{i-1}$.
	
	Further, we adopt the convention that $\partial_{i-k}=0$ when $i-k\leq 0$, and prove by mathematical induction that
	\begin{align}
		\partial_iD^m=\sum_{k=0}^m\binom{m}{k}D^{m-k}\partial_{i-k}.\label{binom}
	\end{align}
	When $m=0$, we have
	\begin{align}
		\partial_iD^0=\partial_i=\sum_{k=0}^0\binom{0}{0}D^0\partial_i=\binom{0}{0}D^0\partial_i.
	\end{align}
	The conclusion holds. Assuming~(\ref{binom}) holds for $m$, then for $m+1$, using the Lie bracket definition $[\partial_i,D]=\partial_i D-D\partial_i$, we obtain
	\begin{align}
		\partial_iD^{m+1}=&(\partial_i D^m)D=\sum_{k=0}^m\binom{m}{k}D^{m-k}\partial_{i-k}D\notag\\
		=&\sum_{k=0}^m\binom{m}{k}D^{m-k}\left(D\partial_{i-k}+[\partial_{i-k},D]\right)\notag\\
		=&\sum_{k=0}^m\binom{m}{k}D^{m-k+1}\partial_{i-k}+\sum_{k=0}^m\binom{m}{k}D^{m-k}\partial_{i-k-1}\notag\\
		=&\binom{m}{0}D^{m+1}\partial_i+\binom{m}{m}D^0\partial_{i-m-1}\notag\\
		&+\sum_{k=1}^m\left(\binom{m}{k}+\binom{m}{k-1}\right)D^{m-k+1}\partial_{i-k}\notag\\
		=&\binom{m+1}{0}D^{m+1}\partial_i+\sum_{k=1}^m\binom{m+1}{k}D^{m-k+1}\partial_{i-k}\notag\\
		&+\binom{m+1}{m+1}D^0\partial_{i-m-1}\notag\\
		=&\sum_{k=0}^{m+1}\binom{m+1}{k}D^{m-k+1}\partial_{i-k}.
	\end{align}
	The induction is complete.
	
	Hence, using the induction result~(\ref{binom}), we have
	\begin{align}
		\partial_iD^mF(x_1)=&\sum_{k=0}^m\binom{m}{k}D^{m-k}\partial_{i-k}F(x_1)\notag\\
		=&\binom{m}{i-1}D^{m-i+1}\partial_1 F(x_1).
	\end{align}
	Thus the case $i-1\leq m$ is proved.
	
	2) Consider the case $i-1>m$. In this case, we only need to prove $\partial_i D^m F(x_1)=0$, then by the convention $\binom{m}{i-1}=0$ when $i-1>m$, formula~(\ref{eq:op_deriv}) still holds. Indeed, by Lemma~\ref{lem-1}, $D^mF(x_1)$ depends only on $x_1,\ldots,x_{m+1}$, and since $i>m+1$, we have
	\begin{align}
		\partial_i D^m F(x_1)=\frac{\partial}{\partial x_{i}}(D^mF)(x_1,\ldots,x_{m+1})=0.
	\end{align}
	The conclusion is proved.
\end{proof}

Below, we prove Theorem~\ref{thm:PH}.
\begin{proof}[Proof of Theorem~\ref{thm:PH}]
	1) First, we prove necessity. Suppose $p(x) = \sum_{m=0}^{n-1} D^m B_m(x_1)$. For each monomial $D^m B_m(x_1)$, applying Lemma~\ref{lem-D} twice yields
	\begin{align}\label{binom1}
		&\partial_i\partial_j D^m B_m(x_1) \notag\\
		= &\binom{m}{i-1}\binom{m-i+1}{j-1} D^{m-i-j+2} B_m''(x_1),
	\end{align}
	where $B_m''(x_1)=\frac{\partial^2B_m(x_1)}{\partial x_1^2}$. Similarly, we have:
	\begin{align}\label{binom2}
		&\partial_1\partial_{i+j-1} D^m B_m(x_1) \notag\\
		=& \binom{m}{0}\binom{m}{i+j-2} D^{m-i-j+2} B_m''(x_1).
	\end{align}
	Using the combinatorial absorption identity
	\begin{align}
		\binom{m}{a}\binom{m-a}{b}=\binom{m}{a+b}\binom{a+b}{a},
	\end{align}
	we obtain
	\begin{align}
		\binom{m}{i-1}\binom{m-i+1}{j-1}=\binom{i+j-2}{i-1}\binom{m}{i+j-2}.
	\end{align}
	Thus, combining~(\ref{binom1}) and~(\ref{binom2}) gives
	\begin{align}\label{binom3}
		\partial_i\partial_j D^m B_m(x_1) = \binom{i+j-2}{i-1} \partial_1\partial_{i+j-1} D^m B_m(x_1).
	\end{align}
	
	Hence, using the linearity of second-order derivatives, we obtain
	\begin{align}
		\partial_i\partial_j p(x) =& \sum_{m=0}^{n-1} \partial_i\partial_j D^m B_m(x_1)
		\notag\\
		=& \binom{i+j-2}{i-1} \sum_{m=0}^{n-1} \partial_1 \partial_{i+j-1} D^m B_m(x_1)
		\notag\\
		=& \binom{i+j-2}{i-1} \partial_1 \partial_{i+j-1} p(x).
	\end{align}
	Note that when $i+j-1 > n$, we have $i+j-2 > n-1$, and since $m \leq n-1$, $\binom{m}{i+j-2}=0$, hence $\partial_1\partial_{i+j-1} D^m B_m(x_1) = 0$. Therefore, condition~(\ref{eq:PH}) holds, i.e., $p(x)$ satisfies the Pascal-Hessian condition.

	2) Next, we prove sufficiency. Suppose $p(x)$ satisfies the Pascal-Hessian condition. Define a residual function
	\begin{equation}\label{eq:res}
		q(x) \triangleq p(x) - \sum_{r=0}^{n-1} D^r B_r(x_1),
	\end{equation}
	where the univariate functions $B_0,\ldots,B_{n-1}$ are given by formulas~(\ref{eq:B0})--(\ref{eq:Br}).
	
	Define a section $\Sigma = \{x\mid x_2=\cdots=x_n=0\}$. First, we show that $p(x)$ and $\partial_i p(x),~i=2,\ldots,n$ are well-defined on $\Sigma$. Actually, setting $i=j=k$ ($k=2,\ldots,n$) in the Pascal-Hessian condition~(\ref{eq:PH}) yields
	\[
	\partial_k^2 p(x) = \binom{2k-2}{k-1}\,\partial_1\partial_{2k-1}p(x)
	\]
	(interpreted as $0$ when $2k-1>n$) as an identity of functions on $\Psi(\mathcal{U})$. If $p(x)$ involved a factor singular in $x_k$, the left-hand side would carry two additional orders of singularity in $x_k$ compared to the right-hand side, making the equality impossible. Hence $p(x)$ and $\partial_i p(x)$ ($i=2,\ldots,n$) are well-defined on $\Sigma$.
	
	Moreover, on $\Sigma$, $D = \sum_{\ell=1}^{n-1} x_{\ell+1}\partial_\ell$ is identically zero (since $x_2=\cdots=x_n=0$), so for any $m\geq 1$ we have $D^m B_m(x_1)|_\Sigma = 0$; and $D^0 B_0(x_1)|_\Sigma = B_0(x_1) = p(x_1,0,\ldots,0) = p(x)|_\Sigma$. Hence $q(x)|_\Sigma = 0$.
	
	Further, we prove that for $r=2,\ldots,n$, $\partial_r q(x)|_\Sigma=0$. By Lemma~\ref{lem-D}, for any $m=0,\ldots,n-1$, we have
	\begin{align}
		\partial_r D^m B_m(x_1) = \binom{m}{r-1} D^{m-r+1} \partial_1 B_m(x_1).
	\end{align}
	For the case $r-1>m$, the right-hand side is $0$. For the case $r-1\leq m$, since when $m-r+1 > 0$, the expression $D^{m-r+1}\partial_1 B_m(x_1)$ necessarily contains at least one factor among $x_2,\ldots,x_n$, it follows that on $\Sigma$, $D^{m-r+1}\partial_1 B_m(x_1)=0$. Hence, only when $m-r+1 = 0$, i.e., $m = r-1$, does $\partial_r D^{r-1} B_{r-1}(x_1)|_\Sigma = \partial_1 B_{r-1}(x_1)\neq 0$ hold. By the definition~(\ref{eq:Br}) of $B_{r-1}(x_1)$, $\partial_1B_{r-1}(x_1) = \partial_r p(x)|_\Sigma$, therefore
	\begin{align}
		\partial_rq(x)\Big|_\Sigma=&\partial_r p(x)|_\Sigma-\partial_r\left(\sum_{\ell=0}^{n-1} D^\ell B_\ell(x_1)\right)\Big|_\Sigma\notag\\
		=&\partial_1 B_{r-1}(x_1)-\partial_1 B_{r-1}(x_1)=0.
	\end{align}
	This shows that $\partial_r q(x)|_\Sigma = 0$.
	
	From~(\ref{binom3}) in the necessity proof, we know that for any $r=0,\ldots,n-1$, $D^r B_r(x_1)$ satisfies the Pascal-Hessian condition; together with the fact that $p(x)$ satisfies the Pascal-Hessian condition, it follows that $q(x)$ also satisfies the Pascal-Hessian condition. Hence, we shall now prove by descending induction that $\partial_r q(x) \equiv 0$ ($r = n, n-1,\ldots, 2$).
	
	For $r=n$: for any $j\geq 2$, we have $n+j-1 > n$, so by the definition of the Pascal-Hessian condition, $\partial_n\partial_j q(x) = 0$, which implies that $\partial_n q(x)$ depends only on $x_1$. Moreover, since $\partial_n q(x)|_\Sigma = 0$ (we have already proved that $\partial_r q(x)|_\Sigma=0$ for $r=2,\ldots,n$), it follows that $\partial_n q(x) \equiv 0$.
	
	Now suppose that for some $r$ ($2 \leq r < n$), $\partial_\ell q(x) \equiv 0$ holds for $\ell = r+1,\ldots,n$. Then, for any $j\geq 2$: if $r+j-1 > n$, the Pascal-Hessian condition directly yields $\partial_r\partial_j q(x) = 0$; conversely, if $r+j-1 \leq n$, then $r+j-1 \geq r+1 > r$, and by the induction hypothesis $\partial_{r+j-1} q(x) \equiv 0$ together with the Pascal-Hessian condition, we still obtain
	\begin{align}
		\partial_r\partial_j q(x)=\binom{r+j-2}{r-1}\partial_1\partial_{r+j-1} q(x) = 0.
	\end{align}
	Therefore, all partial derivatives of $\partial_r q$ with respect to $x_2,\ldots,x_n$ are zero, i.e., $\partial_r q(x)$ depends only on $x_1$. Combining this with $\partial_r q(x)|_\Sigma = 0$ yields $\partial_r q(x) \equiv 0$.
	
	At this point, the mathematical induction is complete, and we have $\partial_2 q(x) \equiv \cdots \equiv \partial_n q(x) \equiv 0$. Hence $q(x)$ is independent of $x_2,\ldots,x_n$, i.e., $q(x) = q(x_1)$. Moreover, $q(x)|_\Sigma = 0$, so $q(x) \equiv 0$. Thus
	\begin{align}
		p(x) = \sum_{r=0}^{n-1} D^r B_r(x_1),
	\end{align}
	i.e., $p(x)$ satisfies the FMS condition~(\ref{fms-p}). Sufficiency is proved.
\end{proof}

The above completes the rigorous proof for the single-output case. Below, we present a systematic comparison between the Pascal-Hessian condition and the traditional OCF necessary and sufficient condition.

\subsection{Comparison Between the Pascal-Hessian Matrix Method and the Traditional OCF Method}\label{sec3.3}

In Sections~\ref{sec3.1} and~\ref{sec3.2}, we have presented the main theorem of the Pascal-Hessian matrix method for single-output nonlinear systems to achieve observer error linearization together with its complete proof. In this subsection, we carry out a detailed comparison between this method and the traditional OCF method. Since Theorem~\ref{thm:PH} is stated for the high-order system~(\ref{sys3}), whereas the traditional OCF method (Lemma~\ref{lem:ocf}) is formulated for~(\ref{sys11})--(\ref{sys12}), we first need to restate Theorem~\ref{thm:PH} in terms of the Pascal-Hessian condition for~(\ref{sys11})--(\ref{sys12}) before making the comparison.
\begin{theorem}\label{thm:gene}
	Given system~(\ref{sys11})--(\ref{sys12}), there exists a nonlinear diffeomorphism $z=\Phi(x)$ such that the system can be transformed into the observer canonical form (or can achieve observer error linearization) on a region $\mathcal{U}$ if and only if: 
	
	1) the codistribution $\Delta^\perp$ has dimension $n$ on $\mathcal{U}$; 
	
	2) the nonlinear function $p(x)=L_f^n(\chi)\big|_{\chi=\Psi^{-1}(x)}$ satisfies the Pascal-Hessian condition on a region $\Psi(\mathcal{U})$, where $x=\Psi(\chi)$ is defined by (\ref{observerform-trans}).
\end{theorem}

\begin{proof}
	1) Necessity. If there exists a diffeomorphism $z=\Phi(x)$ such that system~(\ref{sys11})--(\ref{sys12}) is transformed into the observer canonical form on $\mathcal{U}$, then by Lemma~\ref{lem:ocf}, condition~1) holds and $[\tau_i,\tau_j]=0$ for all $i,j=1,\ldots,n$, where $\tau_i,~i=1,\ldots,n$ are defined in Lemma~\ref{lem:ocf}. By Theorem~\ref{thm:Lie_PH_equiv} in Appendix~B, $[\tau_i,\tau_j]=0$ is equivalent to the Pascal-Hessian condition, hence necessity is proved.
	
	2) Sufficiency. From the condition ``the codistribution $\Delta^\perp$ has dimension $n$ on $\mathcal{U}$'' we know that the diffeomorphism $x=\Psi(\chi)$ exists, from which $p(x)$ is obtained. Further, by the Pascal-Hessian condition, $p(x)$ admits the fully measured form on $\Psi(\mathcal{U})$, i.e., there exists a family of univariate functions $B_r(x_1),~r=0,\ldots,n-1$ defined by~(\ref{eq:B0})--(\ref{eq:Br}) such that~(\ref{fms}) holds, and consequently the OCF~(\ref{sys2}) can be directly obtained via the method provided by Theorem~\ref{thm:PH}.	
\end{proof}

\begin{table}[!t]
\centering
\caption{Full-process comparison between the Pascal-Hessian matrix criterion method and the traditional OCF necessary and sufficient condition solution method}
\label{tab:full-comparison}
\small
\begin{tabularx}{\textwidth}{@{}p{2.5cm}XX@{}}
\toprule
\textbf{Aspect} & \textbf{Pascal-Hessian matrix method} & \textbf{OCF necessary and sufficient condition and solution method} \\
\midrule
Necessary and sufficient condition
&
1) $\dim\Delta^{\perp}=n$\newline
\quad\textbullet\ Compute $L_f h, \ldots, L_f^{(n-1)} h$\newline
2) Pascal-Hessian condition\newline
\quad\textbullet\ Solve $p(x) = L_f^n h(\chi)\vert_{\chi=\Psi^{-1}(x)}$, requiring $n$ derivatives, 1 inner product, 1 nonlinear transformation\newline
\quad\textbullet\ Compute the Hessian matrix of $p(x)$, requiring ${n^2}$ derivatives in total
&
1) $\dim\Delta^{\perp}=n$\newline
\quad\textbullet\ Compute $L_f h, \ldots, L_f^{(n-1)} h$\newline
2) $[\tau_i, \tau_j] = 0,\; i,j=1,\ldots,n$\newline
\quad\textbullet\ Solve the linear system~(\ref{ocf})\newline
\quad\textbullet\ Compute Lie brackets $\tau_k = [\tau_{k-1}, f],\; k=2,\ldots,n$, requiring ${2(n-1)n^2}$ derivatives and ${2(n-1)}$ inner products\newline
\quad\textbullet\ Compute $[\tau_i, \tau_j] = 0,\; i,j=1,\ldots,n$, requiring ${n(n-1) \times n^2}$ derivatives and ${n(n-1)}$ inner products
\\[4pt]
\midrule
Solving OCF
&
1) Compute $n$ indefinite integral~(\ref{eq:B0})--(\ref{eq:Br}).\newline
2) Substitute $a_i(y)=B_{i-1}(y)$ into formula~(\ref{sys2}) for $i=1,\ldots,n$.\newline
\quad\textbullet\ Compute $n$ indefinite integrals.
&
1) Solve the partial differential equations
\[
\frac{\partial \mathcal{F}}{\partial z^\top} = [\tau_1, \ldots, \tau_n]\vert_{\chi=\mathcal{F}(z)}.
\]\newline
2) Find the inverse transformation $z = \Phi(x)$ of $x = \mathcal{F}(z)$.\newline
3) Compute $\frac{\partial \Phi}{\partial x^\top} f(x)\vert_{x=\mathcal{F}(z)}$ and substitute it into equation (\ref{sys2}).\newline
\quad\textbullet\ Solve a system of partial differential equations; Perform one inverse transformation; Perform one nonlinear coordinate transformation.
\\[4pt]
\midrule
Additional cost
&
1) Requires an extra $L_f^n h(\chi)$ ($n$ additional differentiations and one additional inner products).\newline
2) An extra inverse ${\Psi'}^{-1}$ (obtained by trivial back-substitution), resulting in comparable inversion difficulty than the traditional approach.
&
1) An $O(n^4)$ additional differentiations and $O(n^2)$ additional inner products.\newline
2) One extra PDEs solution.
\\
\bottomrule
\end{tabularx}

\smallskip
{\footnotesize Note: The common prerequisites ``$\dim\Delta^{\perp}=n$'' and ``computing $L_f h,\ldots,L_f^{(n-1)}h$'' are retained in both columns for ease of item-by-item comparison.}
\end{table}

In Table~\ref{tab:full-comparison}, we give a item-by-item comparison between Pascal-Hessian condition and traditional OCF theories, from which the advantages of the Pascal-Hessian method are clearly evident.

\textbf{Verification stage.} The traditional OCF method requires sequentially solving a linear system to obtain $\tau_1$, recursively computing $n-1$ Lie brackets $\tau_k=[\tau_{k-1},f]$ (each involving an $n\times n$ Jacobian matrix), and then verifying involutivity for $\binom{n-1}{2}$ pairs of Lie brackets. The total number of derivative operations for condition verification alone reaches $2(n-1)n^2+n(n-1)n^2\approx O(n^4)$, in addition to $O(n^2)$ inner product operations. In contrast, the Pascal-Hessian method only requires solving $p(x)=L_f^nh|_{\chi=\Psi^{-1}(x)}$ and computing its Hessian matrix (an $n\times n$ symmetric matrix), then verifying that its lower-right corner is zero and its upper-left coefficients form Pascal's triangle. The total number of derivative operations is merely $O(n^2)$. The complexity is reduced from $O(n^4)$ to $O(n^2)$, with the difference being particularly pronounced for $n\geq 4$.

\textbf{Construction stage.} The traditional OCF method requires solving $n-1$ coupled first-order PDEs via the Frobenius theorem, which involves cumbersome symbolic operations. The Pascal-Hessian method directly provides the definite integral formulas for $B_r(x_1)$ (Theorem~\ref{thm:PH}). Then, the output injection functions $\alpha_i(y)$ can be obtained by $\alpha_i(y)=B_{i-1}(y)$.


\textbf{Additional cost.} It should be clarified that computing $p(x)$ requires the Lie derivative sequence $L_f h,\ldots,L_f^n h$, whereas verifying condition ($\dim\Delta^\perp=n$) in the traditional approach requires up to $L_f^{n-1}h$. Thus, the additional cost of our method in the preliminary stage consists of two parts: one extra Lie derivative $L_f^n h$, and one nonlinear inverse transformation $\chi=\Psi^{-1}(x)$. For the former, its computational cost is merely $n$ additional derivative and $1$ inner product, which is negligible compared with the subsequent verification of $\binom{n-1}{2}$ pairs of Lie bracket involutivity conditions (involving $O(n^4)$ Lie derivative operations). For the latter, the traditional approach also faces the inversion $\chi=\Phi^{-1}(z)$ at the final stage of OCF construction, whose difficulty is equivalent to that of $\chi=\Psi^{-1}(x)$, and thus these two costs cancel each other. Moreover, the additional coordinate transformation $z=\Psi'(x)$ from the observer form $x$ to the OCF coordinates $z$ (Eqs.~(\ref{transform1})--(\ref{transform2})) has an upper-triangular structure, whose inverse can be obtained explicitly by back-substitution with very little extra effort. In summary, even when the preliminary and subsequent computations are taken together, the overall complexity of our method remains significantly lower than that of the traditional approach.

Summarizing the above two stages, the Pascal-Hessian method reduces the computational cost of condition verification from $O(n^4)$ to $O(n^2)$, and provides explicit integral formulas for $B_r(x_1)$ to avoid solving PDEs.


In the Pascal-Hessian matrix method proposed in this paper, every step from condition verification to the construction of the output injection functions $\alpha(y)$ corresponds to an explicit formula: the condition reduces to checking the Pascal triangle structure of the Hessian matrix; $B_r(x_1)$ is directly obtained from the indefinite integral~(\ref{eq:B0})--(\ref{eq:Br}); and the coordinate transformation is directly written out from~(\ref{transform1})--(\ref{transform2}). These enable the verification and implementation of observer error linearization for high-dimensional nonlinear systems to be carried out following a rigorous formulaic procedure.

\section{Extension of the Pascal-Hessian Condition to a Class of Special Multi-Output Settings}\label{sec4}

This section considers the extension of the Pascal-Hessian condition to the multi-output setting. We assume that the observability indices of all outputs are equal. For the general nonlinear system~(\ref{sys11}), consider the output function
\begin{align}
	\bm{y}=\bm{h}(\chi)=\col\{h_1(\chi),\ldots,h_\omega(\chi)\},~y_i=h_i(\chi),\label{sys13}
\end{align}
where $\bm{y}\in\mathbb{R}^\omega$ is the $\omega$-dimensional measurement output. We assume that under the generalized priority search condition, all outputs have equal observability indices, i.e., the codistribution spanned by the following covector fields
\begin{align}\label{mspan}
	\begin{matrix}
		{\rm d}h_1(\chi)&{\rm d}h_2(\chi)&\cdots&{\rm d}h_\omega(\chi)\\
		{\rm d}L_fh_1(\chi)&{\rm d}L_fh_2(\chi)&\cdots&{\rm d}L_fh_\omega(\chi)\\
		\vdots&\vdots&\ddots&\vdots\\
		{\rm d}L_f^{v-1}h_1(\chi)&{\rm d}L_f^{v-1}h_2(\chi)&\cdots&{\rm d}L_f^{v-1}h_\omega(\chi)\\
	\end{matrix}
\end{align}
has dimension $n$ in the region $\mathcal{U}$, where the observability index $v=n/\omega$. Taking the coordinate transformation $x=\Psi(\chi)$, or more concretely $x_{s_i+j}=L_f^{j}h_i(\chi)$, $i=1,\ldots,\omega$, $j=0,\ldots,v-1$ with $s_i=(v-1)i+1$, we obtain the observer form (written directly as a high-order system)
\begin{align}
	\bm{x}_1^{(v)}=\bm{p}(\bm{x}),~\bm{y}=\bm{x}_1,\label{fms:mo}
\end{align}
where $\bm{p}(\bm{x})=\col\{p_1(\bm{x}),\ldots,p_\omega(\bm{x})\}$, $\bm{x}=\col\{\bm{x}_1,\ldots,\bm{x}_v\}$, and $\bm{x}_j=\col\{x_{s_1+j-1},\ldots,x_{s_\omega+j-1}\}$ for all $j=1,\ldots,v$, . Herein, $s_i$ is the index given by $y_i=x_{s_i}$ with $\bm{y}=\col\{y_1,\ldots,y_\omega\}$. Analogous to the single-output case, we define the total derivative operator for the multi-output case as
\begin{align}
	\bm{D}=\sum_{\ell=1}^{v-1}\bm{x}_{\ell+1}\frac{\partial}{\partial \bm{x}_{\ell}}=\sum_{k=1}^\omega\sum_{\ell=1}^{v-1}x_{s_k+\ell}\frac{\partial}{\partial x_{s_k+\ell-1}}.
\end{align}
We then have the following lemma.
\begin{lemma}\label{lem:D-multi}
	Let $F(\bm{x}_1)$ be a smooth function of $\omega$ variables. For any $m>0$ and $1\leq i\leq n$, we have
	\begin{align}
		\partial_i\bm{D}^mF(\bm{x}_1)=\binom{m}{\hat{i}}\bm{D}^{m-\hat{i}}\frac{\partial F(\bm{x}_1)}{\partial x_{s_{a(i)}}},
	\end{align}
	where $a(i)$ denotes the output index corresponding to state $x_i$, defined as $a(i)=\lceil i/v\rceil$, and the offset $\hat{i}=i-s_{a(i)}$. When $\hat{i}>m$, we adopt the convention that $\binom{m}{\hat{i}}=0$ and $\bm{D}^{m-\hat{i}}=0$. Similarly, for any $1\leq i,j\leq n$, we have
	\begin{align}\label{2order}
		\partial_i\partial_j\bm{D}^mF(\bm{x}_1)=\binom{m}{\hat{i}}\binom{m-\hat{i}}{\hat{j}}\bm{D}^{m-\hat{i}-\hat{j}}\frac{\partial^2 F(\bm{x}_1)}{\partial x_{s_{a(i)}}\partial x_{s_{a(j)}}}.
	\end{align}
\end{lemma}

The proof of this lemma is very similar to that of Lemma~\ref{lem-D} and is omitted. Next, we present the definition of the fully measured system for the multi-output case and the extended form of the Pascal-Hessian condition.
\begin{definition}[Fully measured system]\label{FMS-MO}
	We call the high-order system~(\ref{fms:mo}) fully measured if and only if there exists a family of smooth vector-valued functions $\bm{B}_r(\bm{x}_1)=\col\{B_{1,r}(\bm{x}_1),\ldots,B_{\omega,r}(\bm{x}_1)\}$ such that
	\begin{align}
		\bm{p}(\bm{x})=\sum_{r=0}^{v-1}\bm{B}_{r}^{(r)}(\bm{x}_1)=\sum_{r=0}^{v-1}\bm{D}^r\bm{B}_{r}(\bm{x}_1).
	\end{align}
\end{definition}

\begin{definition}[Pascal-Hessian Condition for Multi-Output Systems]\label{PH:MO}
	The vector-valued function $\bm{p}(\bm{x})$ is said to satisfy the Pascal-Hessian condition if and only if for each $k=1,\ldots,\omega$, the Hessian matrix of $p_k(\bm{x})$ can be partitioned into $\omega\times\omega$ blocks, each of dimension $v\times v$, and within each block, the Pascal triangle structure described in Definition~\ref{def:PH} holds. Concretely, for all $1\leq i,j\leq n$,
	\begin{align}
		\frac{\partial^2 p_k}{\partial x_i \partial x_j} =
		\begin{cases}
			\displaystyle \binom{\hat{i}+\hat{j}}{\hat{i}} \cdot \frac{\partial^2 p_k}{\partial x_{s_{a(i)}} \partial x_{s_{a(j)}+\hat{i}+\hat{j}}}, & \hat{i}+\hat{j}\leq v-1,\\[1.5em]
			0, & \hat{i}+\hat{j}>v-1.
		\end{cases}
	\end{align}
\end{definition}

\begin{figure}[!t]
	\centering
	\includegraphics[width=15cm]{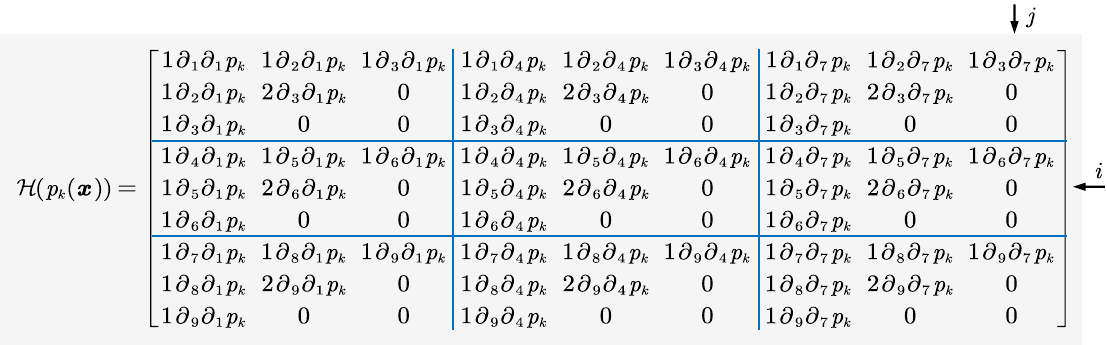}\\
	\caption{Illustration of the multi-output Pascal-Hessian condition.}\label{yanghui2}
\end{figure}

Fig.~\ref{yanghui2} shows an example of $p_k(\bm{x})$ for a system with $9$ states and $3$ outputs (each output having observability index $3$). As can be seen, the entire Hessian matrix can be partitioned into $9$ blocks, each possessing a Pascal triangle structure. We have marked row $i=5$ and column $j=9$ in the figure; we can compute $a(i)=\lceil 5/3\rceil=2$, $a(j)=\lceil 9/3\rceil=3$, hence $\hat{i}=i-s_{a(i)}=1$, $\hat{j}=j-s_{a(j)}=2$; since $\hat{i}+\hat{j}=3>2$, the corresponding entry of the Hessian matrix is $0$.

The following theorem gives the necessary and sufficient conditions for this class of special multi‑output systems to achieve observer error dynamics linearization, as well as the method for constructing smooth vector-valued functions $\bm{B}_r(\bm{x}_1)$.

\begin{theorem}\label{thm:PHm}
	Given the high-order system~(\ref{fms:mo}), let $\bm{x}_{1,j}$ denote the vector obtained by leaving the $1$st--$j$th component of $\bm{x}_1$ unchanged and replacing the other components by zeros. Then there exists a family of smooth vector-valued functions
	\begin{align}
	&\bm{B}_0(\bm{x}_1)=\bm{p}(\bm{x}_1,0,\ldots,0),\label{eq:B0m}\\
	&B_{k,r}(\bm{x}_1)=\sum_{j=1}^{\omega}\int_0^{x_{s_j}}\frac{\partial}{\partial x_{s_j+r}}p_k\left(\bm{x}_{1,j},0,\ldots,0\right)\big|_{x_{s_j}=s}{\rm d}s,\notag\\
	&\quad\qquad\qquad\qquad\qquad\qquad\qquad\qquad\qquad~k=1,\ldots,\omega\label{eq:Brm}
\end{align}
	such that the high-order system~(\ref{fms:mo}) can admit the multi-output high-order fully measured form on $\Psi(\mathcal{U})$ and can consequently be transformed into the observer canonical form (the multi-output observer canonical form and its necessary and sufficient condition in the classical OCF theory are given in Appendix~\ref{app3}) if and only if $\bm{p}(\bm{x})$ satisfies the multi-output Pascal-Hessian condition on $\Psi(\mathcal{U})$.
\end{theorem}

\begin{proof}
	1) Necessity. Suppose $\bm{p}(\bm{x})$ has the FMS form~(\ref{FMS-MO}). For any $k$ and any $i,j$, applying formula~(\ref{2order}) from Lemma~\ref{lem:D-multi} term-by-term to $p_k=\sum_r\bm{D}^r B_{k,r}$, and using the binomial identity $\binom{r}{\hat{\jmath}}\binom{r-\hat{\jmath}}{\hat{\imath}}=\binom{\hat{\imath}+\hat{\jmath}}{\hat{\imath}}\binom{r}{\hat{\imath}+\hat{\jmath}}$, one derives the Pascal proportionality relations in Definition~\ref{PH:MO}. The method is completely parallel to the necessity proof of Theorem~\ref{thm:PH}.
	
	2) Sufficiency. First, for each $k=1,\ldots,\omega$, define the residual function
	\begin{align}
		q_k(\bm{x})=p_k(\bm{x})-\sum_{r=0}^{v-1}\bm{D}^r B_{k,r}(\bm{x}_1),
	\end{align}
	and the section $\Sigma=\{\bm{x}\big| x_i=0,~\forall i\notin \{s_1,\ldots,s_\omega\}\}$. We then prove two claims: (i) $q_k(\bm{x})\big|_\Sigma=0$; and (ii) for all $i=1,\ldots,\omega$ and $r=1,\ldots,v-1$,
	\begin{align}
		\left.\frac{\partial q_k(\bm{x})}{\partial x_{s_i+r}}\right|_{\Sigma}=0.
	\end{align}
	The proof of (i) can be found in the corresponding part of the sufficiency proof of Theorem~\ref{thm:PH}. We now prove (ii). By Lemma~\ref{lem:D-multi}, we obtain
	\begin{align}
		\frac{\partial}{\partial x_{s_i+r}}\bm{D}^mB_{k,m}(\bm{x}_1)=\binom{m}{r}\bm{D}^{m-r}\frac{\partial}{\partial x_{s_i}}B_{k,m}(\bm{x}_1).
	\end{align}
	When $m<r$, the right-hand side is $0$; when $m>r$, since $m-r\neq0$, every term in $\bm{D}^{m-r}\frac{\partial}{\partial x_{s_i}}B_{k,m}(\bm{x}_1)$ contains at least one factor among $\bm{x}_2,\ldots,\bm{x}_n$, hence $\bm{D}^{m-r}\frac{\partial}{\partial x_{s_i}}B_{k,m}(\bm{x}_1)\big|_\Sigma=0$. Thus, only when $m=r$, $\frac{\partial}{\partial x_{s_i+r}}\bm{D}^mB_{k,m}(\bm{x}_1)$ is not zero and gives rise to
	\begin{align}
		\frac{\partial}{\partial x_{s_i+r}}\bm{D}^rB_{k,r}(\bm{x}_1)=\frac{\partial}{\partial x_{s_i}}B_{k,r}(\bm{x}_1)\neq 0.
	\end{align}
	Consequently, from (\ref{eq:B0m})--(\ref{eq:Brm}) we obtain
	\begin{align}
		\left.\frac{\partial q_k(\bm{x})}{\partial x_{s_i+r}}\right|_{\Sigma}
		=\left.\frac{\partial p_k(\bm{x})}{\partial x_{s_i+r}}\right|_{\Sigma}
		-\frac{\partial}{\partial x_{s_i}}B_{k,r}(\bm{x}_1).
	\end{align}
	Thus, proving claim~(ii) is equivalent to verifying that $B_{k,r}(\bm{x}_1)$ defined in~(\ref{eq:Brm}) satisfies the gradient relation
	\begin{align}
		\frac{\partial B_{k,r}}{\partial x_{s_i}}(\bm{x}_1)
		=\left.\frac{\partial p_k}{\partial x_{s_i+r}}\right|_\Sigma,\qquad i=1,\ldots,\omega. \label{gradient}
	\end{align}
	The section $\Sigma$ is diffeomorphic to $\mathbb{R}^\omega$, hence simply connected. By the Poincar\'e lemma, the vector field on the right-hand side of~(\ref{gradient}) is a gradient field if and only if it is curl-free, i.e., for all $i,j=1,\ldots,\omega$,
	\begin{align}
		\frac{\partial}{\partial x_{s_j}}\!\left(\frac{\partial p_k}{\partial x_{s_i+r}}\Big|_\Sigma\right)
		=\frac{\partial}{\partial x_{s_i}}\!\left(\frac{\partial p_k}{\partial x_{s_j+r}}\Big|_\Sigma\right). \label{Poincare}
	\end{align}
	We now verify~(\ref{Poincare}) using the multi-output Pascal-Hessian condition. The left-hand side equals $\frac{\partial^2 p_k}{\partial x_{s_j}\partial x_{s_i+r}}\big|_\Sigma$, which by Clairaut's theorem is $\frac{\partial^2 p_k}{\partial x_{s_i+r}\partial x_{s_j}}\big|_\Sigma$. This Hessian entry lies in block $(i,j)$ with $(\hat{i},\hat{j})=(r,0)$. Since $\hat{i}+\hat{j}=r\leq v-1$, the Pascal-Hessian condition~(\ref{PH:MO}) gives
	\[
	\frac{\partial^2 p_k}{\partial x_{s_i+r}\partial x_{s_j}}
	=\binom{r}{r}\,\frac{\partial^2 p_k}{\partial x_{s_i}\partial x_{s_j+r}}
	=\frac{\partial^2 p_k}{\partial x_{s_i}\partial x_{s_j+r}},
	\]
	which is precisely the right-hand side of~(\ref{Poincare}). Hence~(\ref{Poincare}) holds, the curl-free condition is satisfied, and claim~(ii) follows from the Poincar\'e lemma.
	
	Based on claims (i)--(ii), and following the descending induction argument in the sufficiency proof of Theorem~\ref{thm:PH}, it can be proved that $q_k(\bm{x}_1)\equiv0$. This completes the sufficiency proof.
\end{proof}

As in the single-output case, Theorem~\ref{thm:PHm} provides both verification and construction methods. For verification, one computes the Hessian matrix of each $p_k(\bm{x})$, partitions it into $\omega\times\omega$ blocks of size $v\times v$ according to Definition~\ref{PH:MO}, and checks the Pascal triangle structure block by block. For construction, $\bm{B}_r(\bm{x}_1)$ is directly obtained from the indefinite integral~(\ref{eq:B0m})--(\ref{eq:Brm}) on the section $\Sigma$; following the single-output procedure yields the OCF coordinate transformation and output injection functions.

\begin{remark}
This paper only discusses the case where the observability indices of all outputs are equal ($v=n/\omega$). This case is very common in engineering: for example, in interconnected microgrid systems, the dynamic orders of each generation unit are the same; in multi-UAV/multi-vehicle formations, the kinematic models of homogeneous platforms are identical; in coupled Van~der~Pol oscillator networks, the state dimension of each oscillator is equal. The results of this section can be directly applied to all such systems. 
\end{remark}
\begin{remark}
	Extending the proposed method to multi-output systems with unequal observability indices is a meaningful direction for future research. The unequal orders of Lie derivatives across different output channels, caused by the disparity in observability indices, render the block structure of the Hessian matrix non-uniform, which constitutes the main technical difficulty in this generalization. We are currently conducting systematic investigations on viable approaches to this extension, and preliminary explorations have shown promising prospects. The corresponding results will be reported elsewhere in due course.
\end{remark}
\section{Computational Examples}\label{sec:example}
In this section, two examples are provided from complementary perspectives. Subsection \ref{sec5.1} uses a fourth‑order numerical system to illustrate the single‑output procedure step by step and compares it with the traditional approach. Subsection \ref{sec5.2} then applies the multi‑output extension to a spacecraft relative‑motion system.

\subsection{Application Example of the Pascal-Hessian Matrix Method on a Numerical System}\label{sec5.1}

Consider the following fourth-order nonlinear system
	\begin{align}
		\dot{x}_1&=x_2,\label{eq:ex-ss1}\\
		\dot{x}_2&=x_3,\\
		\dot{x}_3&=x_4,\\
		\dot{x}_4&=x_1x_3+x_1x_4+x_2^2+3x_2x_3+x_1^3,\label{eq:ex-ss2}\\
		y&=x_1.
	\end{align}
This system is already in the observer form~(\ref{eq:chain}), with $p(x)=x_1x_3+x_1x_4+x_2^2+3x_2x_3+x_1^3$ and $h(x)=x_1$. Below, we solve for its OCF using both the Frobenius integral method and the Pascal-Hessian matrix method.

\textbf{Frobenius integral method.}
Solving ${\rm d}L_f^{r}h\cdot\tau_1=\delta_{r,3}$ yields $\tau_1=\col\{0,0,0,1\}$. Using the recursion $\tau_{k+1}=[\tau_k,f]$ together with $f_p(x)=\col\{x_2,x_3,x_4,p(x)\}$ and $\partial f_p/\partial x$, successive computation yields
\begin{align*}
	\tau_1&=\col\{0,\;0,\;0,\;1\},\\
	\tau_2&=\col\{0,\;0,\;1,\;x_1\},\\
	\tau_3&=\col\{0,\;0,\;x_1,\;x_1+2x_2+x_1^2\},\\
	\tau_4&=\col\{0,\;0,\;x_1+x_2+x_1^2,\notag\\
	&\qquad\qquad\qquad2x_1^2+3x_1x_2+x_1^3-x_2-2x_3\}.
\end{align*}
Pairwise verification gives $[\tau_i,\tau_j]=0$ for $i,j=1,\ldots,4$; the involutivity condition is satisfied. According to Appendix~\ref{app1}, one must solve the PDE
\begin{equation}\label{eq:ex-PDE}
	\frac{\partial\mathcal{F}}{\partial z^\top}=[\tau_1,\tau_2,\tau_3,\tau_4]\big|_{x=\mathcal{F}(z)}.
\end{equation}
Solving this PDE depends critically on the concrete expressions of the $\tau_k$---the nonzero entries of the right-hand side matrix involve $x_1,x_2,x_3$ as well as cross terms such as $x_1^2,x_1x_2,x_1^3$; the characteristic equations couple the integration paths of $\mathcal{F}_3,\mathcal{F}_4$, and changing $p(x)$ requires re-deriving a new set of characteristic equations, with no universal closed-form formula independent of the system. After obtaining $\mathcal{F}$, one must further invert $x=\mathcal{F}(z)$. Every step is anchored to the specific form of the $\tau_k$, in contrast to the Pascal-Hessian method where all steps proceed from $p(x)$ directly to $B_r(\cdot)$ and then to the OCF via fixed formulas.

\textbf{Pascal-Hessian matrix method.}
From $p(x)=x_1x_3+x_1x_4+x_2^2+3x_2x_3+x_1^3$, compute the second-order partial derivatives directly:
\begin{align}
	\partial_1\partial_1p=6x_1,\;\;
	\partial_1\partial_2p=0,\;\;
	\partial_1\partial_3p=1,\;\;
	\partial_1\partial_4p=1,\notag\\
	\partial_2\partial_2p=2,\;\;
	\partial_2\partial_3p=3,\;\;
	\partial_2\partial_4p=0,\notag\\
	\partial_3\partial_3p=0,\;\;
	\partial_3\partial_4p=0,\;\;
	\partial_4\partial_4p=0.\notag
\end{align}
Substituting into the Hessian matrix gives
\begin{equation}\label{eq:ex-H}
	\mathcal{H}(p)=\begin{bmatrix}
		6x_1 & 0 & 1 & 1\\
		0    & 2 & 3 & 0\\
		1    & 3 & 0 & 0\\
		1    & 0 & 0 & 0
	\end{bmatrix}.
\end{equation}
The lower-right of the anti-diagonal is identically zero, and the upper-left corner satisfies $\mathcal{H}_{22}=2=2\mathcal{H}_{13}$, $\mathcal{H}_{23}=3=3\mathcal{H}_{14}$; the Pascal-Hessian condition holds. The contrast is clear: computing the Pascal-Hessian condition is far simpler than computing numerous Lie brackets to verify the involutivity condition $[\tau_i,\tau_j]=0$.

Furthermore, by Theorem~\ref{thm:PH}, the $B_r$ are obtained by integration on the section $\Sigma$:
$B_0(x_1)=x_1^3$, $B_1(x_1)=0$, $B_2(x_1)=B_3(x_1)=x_1^2/2$.
Using~(\ref{transform1})--(\ref{transform2}), the OCF coordinate transformation (with $D=x_2\partial_1+x_3\partial_2+x_4\partial_3$) is:
\begin{align}
	z_4&=x_1,\quad
	z_3=x_2-\frac{x_1^2}{2},\quad
	z_2=x_3-\frac{x_1^2}{2}-x_1x_2,\notag\\
	z_1&=x_4-x_1x_2-x_2^2-x_1x_3,\label{eq:ex-OCF}
\end{align}
and the output injection functions are $\alpha_1(y)=y^3$, $\alpha_2(y)=0$, $\alpha_3(y)=y^2/2$, $\alpha_4(y)=y^2/2$. The observer is
\begin{equation}\label{eq:ex-observer}
	\dot{\hat{z}}=A\hat{z}+\alpha(y)+L(y-\hat{z}_4),\qquad C=[0,0,0,1],
\end{equation}
with error dynamics $\dot{\hat{e}}=(A-LC)\hat{e}$ converging asymptotically. The transformation~(\ref{eq:ex-OCF}) has an upper-triangular structure; its inverse can be obtained row by row via back-substitution to recover $\hat{x}$ from $\hat{z}$. From~(\ref{eq:ex-ss1})--(\ref{eq:ex-ss2}), the system is already in observer form, hence $\hat{\chi}=\hat{x}$.

The above comparison shows that, for the same system, every step of the Frobenius integral method---$\tau_k$ recursion, PDE solving, inversion---is tightly coupled to the concrete forms of $f_p$ and $p$ for that system; changing $p(x)$ requires re-running the entire procedure, and the PDE-solving step lacks a universal formula. In contrast, the Pascal-Hessian method unifies the condition verification as a binomial-coefficient check of Hessian entries; the $B_r(\cdot)$ are obtained at once via definite integral formulas; and the OCF coordinate transformation together with state recovery is executed following a fixed procedure, independent of the concrete functional form of $p(x)$.

\subsection{Multi-Output Physical Example: Spacecraft Relative Motion System}\label{sec5.2}

This subsection uses a spacecraft relative motion system as an example to verify the effectiveness of the multi-output Pascal-Hessian condition in Section~\ref{sec4}. Consider a spacecraft (without control input) moving on a circular reference orbit. Its relative motion dynamics in the Reference Orbit Coordinate System (ROCS) can be described as
\begin{equation}
	\dot{\rho}=v,\qquad 
	\dot{v}=-\frac{1}{m}\mathbf{C}(\omega_n)v-\frac{1}{m}\mathbf{N}(\rho,\omega_n,R_e),
\end{equation}
where $\rho=\mathrm{col}\{\rho_1,\rho_2,\rho_3\}\in\mathbb{R}^3$ and $v=\mathrm{col}\{v_1,v_2,v_3\}\in\mathbb{R}^3$ are, respectively, the position and velocity of the spacecraft relative to the origin of the reference orbit coordinate system, $m$ is the spacecraft mass, $\omega_n=\sqrt{GM/R_e^3}$ is the reference orbit angular velocity, and
\begin{align}
	\mathbf{C}(\omega_n)&=2\omega_n\begin{bmatrix}
		0&-1&0\\
		1&0&0\\
		0&0&0
	\end{bmatrix},\\
	\mathbf{N}(\rho,\omega_n,R_e)&=\begin{bmatrix}
		GM\left(\dfrac{\rho_1+R_e}{r^3}-\dfrac{1}{R_e^2}\right)-\omega_n^2\rho_1\\[6pt]
		GM\dfrac{\rho_2}{r^3}-\omega_n^2\rho_2\\[6pt]
		GM\dfrac{\rho_3}{r^3}
	\end{bmatrix},
\end{align}
with $r=\left((R_e+\rho_1)^2+\rho_2^2+\rho_3^2\right)^{1/2}$, $G$ the gravitational constant, $M$ the Earth mass, and $R_e$ the reference orbit radius.

Take the state variables $$x=[x_{11},x_{12},x_{21},x_{22},x_{31},x_{32}]^\top=[\rho_1,v_1,\rho_2,v_2,\rho_3,v_3]^\top$$ and the output $y=[y_1,y_2,y_3]^\top=[\rho_1,\rho_2,\rho_3]^\top$. The system is already in observer form ($n=6$, $\omega=3$, $v=2$):
\begin{equation}
	\dot{x}_{k1}=x_{k2},\qquad 
	\dot{x}_{k2}=p_k(x),\quad k=1,2,3,
\end{equation}
where $p_k(x)$ is the $k$-th component of the vector-valued function $p(x)=\mathrm{col}\{p_1(x),p_2(x),p_3(x)\}$, and
\begin{equation}
	p_k(\rho,v)=-\frac{1}{m}\left(\mathbf{C}(\omega_n)v\right)_k-\frac{1}{m}\mathbf{N}_k(\rho,\omega_n,R_e),
\end{equation}
where $\left(\mathbf{C}(\omega_n)v\right)_k$ denotes the $k$-th component of the vector $\mathbf{C}(\omega_n)v$, and $\mathbf{N}_k(\rho,\omega_n,R_e)$ denotes the $k$-th component of the vector $\mathbf{N}(\rho,\omega_n,R_e)$.

Now verify the Pascal-Hessian condition. Note that $\mathbf{N}(\rho,\omega_n,R_e)$ depends only on the position $\rho$, while $\mathbf{C}(\omega_n)v$ is linear in $v$, so its second-order partial derivatives with respect to the velocity components are zero. Hence, for each $k=1,2,3$, all second-order partial derivatives of $p_k(x)$ with respect to the velocity components are zero; the second-order partial derivatives with respect to the position components come from $\mathbf{N}_k(\rho,\omega_n,R_e)$, and under the block coordinate partition, they naturally satisfy the multi-output Pascal-Hessian condition in Definition~\ref{PH:MO}. Therefore, the system can be directly transformed into the multi-output observer canonical form.

By Theorem~\ref{thm:PHm}, compute $B_{k,r}(y)$ on the section $\Sigma=\{v=0\}$. For $k=1,2,3$,
\begin{align*}
	B_{k,0}(y)=&p_k(y,0)=-\frac{1}{m}\mathbf{N}_k(y,\omega_n,R_e),\\
	B_{k,1}(y)=&0.
\end{align*}
Thus, the multi-output observer canonical form corresponding to the system is $\dot{z}=Az+\alpha(y)$, where $\alpha_{k,0}=B_{k,0}(y)$, $\alpha_{k,1}=B_{k,1}(y)=0$, and $A$ is the block shift-down matrix. Introducing an observer gain $L$ such that $A-LC$ is Hurwitz, we obtain an observer with linear error dynamics:
\begin{equation}
	\dot{\hat{z}}=A\hat{z}+\alpha(y)+L(y-C\hat{z}).
\end{equation}

This example demonstrates that the multi-output Pascal-Hessian condition can be effectively applied to a physically motivated spacecraft relative motion system, verifying the engineering applicability of the theoretical results in Section~\ref{sec4}.

\section{Conclusion and Outlook}\label{sec6}

This paper, targeting the classical theory of nonlinear observer error linearization, has proposed the Pascal-Hessian condition and established a complete framework for verification and construction. On the verification side, this condition equivalently converts the intricate Lie bracket involutivity tests of the traditional OCF theory into an algebraic test on the Hessian matrix structure of the system's nonlinear term, reducing the computational complexity from $O(n^4)$ to $O(n^2)$. On the construction side, we provide explicit integral formulas for all output-dependent univariate functions in the canonical form, together with a direct construction method for the observer canonical form, thereby avoiding the solution of PDE systems and nontrivial nonlinear transformations. We further extend these results to multi-output systems with equal observability indices, giving the corresponding block Pascal triangle condition. This work reveals the fundamental mathematical connection among the theory of observer error linearization, high-order fully measured systems, and Pascal's triangle. The proposed method is entirely formulaic: every step can be executed according to explicit rules, endowing the verification and design of this theory for high-dimensional systems with a feasible implementation pathway.

Future work will focus on extending the Pascal-Hessian matrix method to the general multi-output case with unequal observability indices, and further considering affine nonlinear systems, reduced-order observer settings, and augmented-state observer settings based on dynamic extension.

\appendix
\section*{\centering\LARGE Appendix}
\addcontentsline{toc}{section}{Appendix}
\makeatletter
\renewcommand{\section}{\@startsection{section}{1}{\z@}%
	{-3.5ex \@plus -1ex \@minus -.2ex}%
	{2.3ex \@plus.2ex}%
	{\normalfont\large\bfseries}}
\makeatother

\section{The Complete Computational Procedure of the Traditional OCF Method}\label{app1}

Lemma~\ref{lem:ocf} in the main text has given the necessary and sufficient condition for system~(\ref{sys11})--(\ref{sys12}) to be transformable into the observer canonical form~(\ref{sys2}). This section, under the premise that the necessary and sufficient condition is satisfied, presents two classical methods for constructing the OCF coordinate transformation $z=\Phi(\chi)$ and the output injection functions $\alpha_i(y)$.


By the Frobenius theorem, the necessary and sufficient condition $[\tau_i,\tau_j]=0,~i,j=1,\ldots,n$, for observer error linearization is equivalent to the PDE
\begin{align}
	\frac{\partial \mathcal{F}}{\partial z^\top} = [\tau_1, \ldots, \tau_n]\vert_{\chi=\mathcal{F}(z)}.
\end{align}
Solving this PDE yields the nonlinear diffeomorphism $\chi=\mathcal{F}(z)$. Further, by inverting to obtain $z=\Phi(\chi)$ (the inverse always exists because $\chi=\mathcal{F}(z)$ is a diffeomorphism), the transformed system is
\begin{align}
	\dot{z}=\frac{\partial \Phi(\chi)}{\partial \chi^\top}f(\chi)\big|_{\chi=\mathcal{F}(z)},
\end{align}
which possesses the form of the OCF standard form~(\ref{sys2}).

\section{Equivalence Between the Pascal-Hessian Condition and $[\tau_i,\tau_j]=0$}\label{app2}

This appendix serves as the proof basis for the necessity part of Theorem~\ref{thm:gene} in Section~\ref{sec3.3} of the main text. Specifically, we prove that the Pascal-Hessian condition given by Theorem~\ref{thm:PH} is equivalent to the Lie bracket involutivity condition $[\tau_i,\tau_j]=0$ of the classical OCF theories.

%

The following lemma is a direct corollary of the Frobenius theorem, known in differential geometry as the Commuting Frame Rectification Theorem \cite{Lee2012SmoothManifolds}.

\begin{lemma}[Commuting Frame Rectification Theorem]\label{lem:straightening}
	Let $\tau_1,\ldots,\tau_n$ be smooth vector fields on an $n$-dimensional smooth manifold that are everywhere linearly independent and pairwise commuting (i.e., $[\tau_i,\tau_j]=0$, $\forall i,j$). Then there exists a local coordinate chart $s=(s_1,\ldots,s_n)$ such that $\tau_i=\frac{\partial}{\partial s_i}$ for each $i=1,\ldots,n$.
\end{lemma}

\begin{theorem}\label{thm:Lie_PH_equiv}
	In the observer form~(\ref{eq:chain}), $[\tau_i,\tau_j]=0$ for all $i,j=1,\ldots,n$ is equivalent to $p(x)$ satisfying the Pascal-Hessian condition.
\end{theorem}

\begin{proof}
	We prove the two directions separately.
	
	\noindent (1) Pascal-Hessian condition $\Rightarrow [\tau_i,\tau_j]=0$.

	By the sufficiency part of Theorem~\ref{thm:PH}, the Pascal-Hessian condition is equivalent to the existence of a family of smooth univariate functions $B_0(x_1),\ldots,B_{n-1}(x_1)$ such that
	\begin{equation}\label{eq:p_sum}
		p(x)=\sum_{r=0}^{n-1}D^rB_r(x_1).
	\end{equation}
	Define the coordinate transformation $x\mapsto z$ as~(\ref{transform1})--(\ref{transform2}); its Jacobian is anti-triangular with unit diagonal, hence a local diffeomorphism.

	Since $\dot{B}(x_1)=DB(x_1)$ for any $B(x_1)$, whence
	\[
	\frac{\mathrm{d}}{\mathrm{d}t}D^{r-k}B_r(x_1)=D^{r-k+1}B_r(x_1).
	\]
	Verifying the dynamics in $z$-coordinates: for $k=1$,
	\begin{align}
		\dot{z}_1=&\dot{x}_n-\sum_{r=1}^{n-1}\frac{\mathrm{d}}{\mathrm{d}t}D^{r-1}B_r(x_1)\notag\\
		=&p-\sum_{r=1}^{n-1}D^rB_r(x_1)=B_0(x_1)=B_0(z_n).
	\end{align}
	For $k\geq2$,
	\[
	\begin{aligned}
		\dot{z}_k&=x_{n-k+2}-\sum_{r=k}^{n-1}D^{r-k+1}B_r(x_1)\\
		&=z_{k-1}+B_{k-1}(z_n).
	\end{aligned}
	\]
	Therefore, the vector field $f_p$ in $z$-coordinates can be expressed as
	\begin{equation}\label{eq:fp_s}
		f_p=B_0(s_n)\frac{\partial}{\partial z_1}+\sum_{k=2}^{n}\left(z_{k-1}+B_{k-1}(z_n)\right)\frac{\partial}{\partial z_k}.
	\end{equation}

	Now compute the $\tau$ vector fields in $z$-coordinates. In $x$-coordinates, $\tau_1=e_n=\partial/\partial x_n$. From the dependence relations of the terms in~(\ref{transform1})--(\ref{transform2}) (Lemma~\ref{lem-1} guarantees that $\sum_{r=k}^{n-1}D^{r-k}B_r(x_1)$ depends only on $x_1,\ldots,x_{n-1}$), we obtain
	\[
	\tau_1s_1=1,\qquad \tau_1s_k=0\;(k\geq2).
	\]
	Hence, in $s$-coordinates, $\tau_1=\partial/\partial z_1$.

	From~(\ref{eq:fp_s}), direct computation yields
	\[
	\left[\frac{\partial}{\partial z_i},f_p\right]=\frac{\partial}{\partial z_{i+1}},\qquad i=1,\ldots,n-1.
	\]
	Furthermore, since $\tau_{i+1}=[\tau_i,f_p]$ and the Lie bracket is invariant under coordinate transformations, induction yields
	\[
	\tau_i=\frac{\partial}{\partial z_i},\qquad i=1,\ldots,n.
	\]
	Therefore, for any $i,j$,
	\[
	[\tau_i,\tau_j]=\left[\frac{\partial}{\partial z_i},\frac{\partial}{\partial z_j}\right]=0.
	\]
	The forward direction is proved.

	\medskip
	\noindent (2) $[\tau_i,\tau_j]=0\Rightarrow$ Pascal-Hessian condition.

	Assume $[\tau_i,\tau_j]=0$ for all $i,j$. By Lemma~\ref{lem:straightening}, there exist local coordinates $s=(s_1,\ldots,s_n)$ such that
	\begin{equation}\label{eq:tau_s}
		\tau_i=\frac{\partial}{\partial s_i},\qquad i=1,\ldots,n.
	\end{equation}
	From $\tau_i x_1=\partial x_1/\partial s_i$ together with $\tau_n x_1=1$ and $\tau_i x_1=0$ ($i<n$), we obtain $\partial x_1/\partial s_i=\delta_{in}$. By a suitable additive constant shift, we may take
	\begin{equation}\label{eq:sn_x1}
		s_n=x_1.
	\end{equation}

	In $s$-coordinates, express $f_p$ as
	\[
	f_p=\sum_{j=1}^n a_j(s)\frac{\partial}{\partial s_j}.
	\]
	From the recursion $\tau_{i+1}=[\tau_i,f_p]$ and~(\ref{eq:tau_s}), we obtain
	\[
	\frac{\partial}{\partial s_{i+1}}=\left[\frac{\partial}{\partial s_i},f_p\right]=\sum_{j=1}^n\frac{\partial a_j}{\partial s_i}\frac{\partial}{\partial s_j},\qquad i=1,\ldots,n-1.
	\]
	Comparing coefficients gives
	\begin{equation}\label{eq:da}
		\frac{\partial a_j}{\partial s_i}=\delta_{j,i+1},\qquad i=1,\ldots,n-1.
	\end{equation}
	Integrating~(\ref{eq:da}) successively: for $j=1$, $\partial a_1/\partial s_i=0$ ($\forall i\leq n-1$), hence $a_1=B_0(s_n)$; for $j=2$, $\partial a_2/\partial s_1=1$, $\partial a_2/\partial s_i=0$ ($i\geq2$), and integrating gives $a_2=s_1+B_1(s_n)$. Proceeding analogously, we obtain
	\begin{align}
		a_1=&B_0(s_n),\\
		a_j=&s_{j-1}+B_{j-1}(s_n),\quad j=2,\ldots,n,\label{eq:a_j}
	\end{align}
	where $B_0(s_n),\ldots,B_{n-1}(s_n)$ are smooth functions depending only on $s_n=x_1$.

	Now return from $s$-coordinates to $x$-coordinates. In the observer form~(\ref{eq:chain}), we have $x_{m+1}=L_{f_p}^m x_1$ ($m=0,\ldots,n-1$). Back-substituting step by step from~(\ref{eq:a_j}): first $x_1=s_n$. From $\dot{s}_n=s_{n-1}+B_{n-1}(x_1)=\dot{x}_1=x_2$, we obtain
	\[
	x_2=s_{n-1}+B_{n-1}(x_1).
	\]
	Continuing recursively, generally we have
	\begin{equation}\label{eq:x2s_inv}
		x_{n-k+1}=s_k+\sum_{r=k}^{n-1}L_{f_p}^{r-k}B_r(x_1),\qquad k=1,\ldots,n.
	\end{equation}
	Taking $k=1$,
	\[
	x_n=s_1+\sum_{r=1}^{n-1}L_{f_p}^{r-1}B_r(x_1).
	\]
	Applying $L_{f_p}$ to the above equality and using $\dot{s}_1=a_1=B_0(x_1)$,
	\begin{equation}\label{eq:p_Lf}
		p(x)=L_{f_p}x_n=B_0(x_1)+\sum_{r=1}^{n-1}L_{f_p}^r B_r(x_1).
	\end{equation}

	Finally, we prove that for any univariate function $B(x_1)$ and $r\leq n-1$,
	\begin{equation}\label{eq:Lf_D}
		L_{f_p}^rB(x_1)=D^rB(x_1).
	\end{equation}
	The case $r=0$ is obvious. Assume $L_{f_p}^rB(x_1)=D^rB(x_1)$ depends only on $x_1,\ldots,x_{r+1}$ (by Lemma~\ref{lem-1}). If $r\leq n-2$, then $r+1\leq n-1$, so $\partial_n\left(L_{f_p}^rB(x_1)\right)=0$. From $f_p=D+p\partial_n$, we obtain
	
	\begin{align}
		L_{f_p}^{r+1}B(x_1)&=(D+p\partial_n)\left(L_{f_p}^rB(x_1)\right)\notag\\
		&=D\left(D^rB(x_1)\right)+p\cdot0=D^{r+1}B(x_1).
	\end{align}
	Induction up to $r=n-2$ already covers all cases with $r\leq n-1$, hence~(\ref{eq:Lf_D}) holds.

	Substituting~(\ref{eq:Lf_D}) into~(\ref{eq:p_Lf}) gives
	\[
	p(x)=\sum_{r=0}^{n-1}D^rB_r(x_1).
	\]
	By Theorem~\ref{thm:PH}, the above is equivalent to the Pascal-Hessian condition. The backward direction is proved.

	In summary, Theorem~\ref{thm:Lie_PH_equiv} is established, i.e., the Pascal-Hessian condition and the Lie bracket involutivity condition $[\tau_i,\tau_j]=0$ are fully equivalent. This result demonstrates that the algebraic criterion proposed in this paper is in complete theoretical accord with the traditional differential geometric approach.
\end{proof}

\section{Multi-Output Observer Canonical Form}\label{app3}

This appendix gives the concrete form of the observer canonical form corresponding to the multi-output systems in Section~\ref{sec4}. Consider a nonlinear system with $\omega$-dimensional output and state dimension $n=\omega v$. The multi-output observer canonical form has the following block structure:
\begin{equation}
	\dot{z} = A z + \alpha(y), \qquad y = C z,
\end{equation}
where $z=\mathrm{col}\{z_1,\ldots,z_\omega\}\in\mathbb{R}^n$, each sub-block $z_i\in\mathbb{R}^{v}$, $y=\mathrm{col}\{y_1,\ldots,y_\omega\}\in\mathbb{R}^\omega$, $\alpha(y)=\mathrm{col}\{\alpha_1(y),\ldots,\alpha_\omega(y)\}$ is a smooth vector-valued function, and
\[
A = I_\omega\otimes A_0,
\qquad
C = I_\omega\otimes C_0,
\]
where $\otimes$ denotes the Kronecker product, and each diagonal block $A_0\in\mathbb{R}^{v\times v}$ and $C_0\in\mathbb{R}^{1\times v}$ are given by
\begin{equation}
	A_0 = \begin{bmatrix}
		0 & \cdots & 0 &0\\
		1 & \cdots & 0 &0\\
		\vdots & \ddots & \vdots & \vdots \\
		0 & \cdots & 1 &0
	\end{bmatrix},\qquad
	C_0 = \begin{bmatrix} 0 & \cdots & 0 & 1 \end{bmatrix}.
\end{equation}

For the multi-output nonlinear system~(\ref{sys11}),~(\ref{sys13}), the necessary and sufficient condition for it to be transformable into the observer canonical form on a region $\mathcal{U}$ is that the codistribution spanned by~(\ref{mspan}) has dimension $n$ on $\mathcal{U}$, and that the involutivity condition $[\tau_{ik},\tau_{jk}]=0$ holds for all $i,j=1,\ldots,\omega$ and $k=1,\ldots,v$, where $\tau_{i1}$ is the solution of the linear system $\langle {\rm d}L_f^{\ell}h_i(\chi),\tau_{i1}\rangle=\delta_{\ell,v-1}$, and $\tau_{i\ell}=[\tau_{i(\ell-1)},f]$.

\bibliographystyle{unsrt}
\bibliography{reference}

\vfill

\end{document}